\documentclass[psamsfonts]{amsart}
\usepackage{amssymb,amsfonts}
\usepackage{latexsym}
\usepackage{enumerate}
\usepackage{mathrsfs}
\usepackage{url}
\usepackage{algorithmic}
\usepackage{algorithm}
\usepackage{tikz}

\newtheorem{thm}{Theorem}[section]
\newtheorem{cor}[thm]{Corollary}
\newtheorem{prop}[thm]{Proposition}
\newtheorem{lem}[thm]{Lemma}
\newtheorem{quest}[thm]{Question}

\newtheorem{conj}{Conjecture}

\theoremstyle{definition}
\newtheorem{defn}[thm]{Definition}
\newtheorem{defns}[thm]{Definitions}

\theoremstyle{remark}

\numberwithin{equation}{section}

\title{Integer Complexity: Algorithms and Computational Results}
\author{Harry Altman}
\date{December 20, 2017}

\begin{document}

\newcommand{\cpx}[1]{\|#1\|}
\newcommand{\dft}{\delta}
\newcommand{\st}{{st}}
\newcommand{\xpdd}[1]{\hat{#1}}
\newcommand{\drop}{\Delta}
\newcommand{\badfac}{\kappa}
\newcommand{\cR}{R}
\newcommand{\acl}{\ell}

\newcommand{\NUM}{{48}}
\newcommand{\TWONUM}{{96}}

\newcommand{\N}{{\mathbb N}}
\newcommand{\R}{{\mathbb R}}
\newcommand{\Z}{{\mathbb Z}}
\newcommand{\Q}{{\mathbb Q}}
\newcommand{\sS}{{\mathcal S}}
\newcommand{\sT}{{\mathcal T}}

\newcommand{\floor}[1]{{\lfloor #1 \rfloor}}
\newcommand{\ceil}[1]{{\lceil #1 \ceil}}

\begin{abstract}
Define $\cpx{n}$ to be the \emph{complexity} of $n$, the smallest number of ones
needed to write $n$ using an arbitrary combination of addition and
multiplication.  Define $n$ to be \emph{stable} if for all $k\ge 0$, we have
$\cpx{3^k n}=\cpx{n}+3k$.  In \cite{paper1}, this author and Zelinsky showed
that for any $n$, there exists some $K=K(n)$ such that $3^K n$ is stable;
however, the proof there provided no upper bound on $K(n)$ or any way of
computing it.  In this paper, we describe an algorithm for computing $K(n)$, and
thereby also show that the set of stable numbers is a computable set.  The
algorithm is based on considering the \emph{defect} of a number, defined by
$\dft(n):=\cpx{n}-3\log_3 n$, building on the methods presented in
\cite{theory}.  As a side benefit, this algorithm also happens to allow fast
evaluation of the complexities of powers of $2$; we use it to verify that
$\cpx{2^k 3^\ell}=2k+3\ell$ for $k\le\NUM$ and arbitrary $\ell$ (excluding the
case $k=\ell=0$), providing more evidence for the conjecture that $\cpx{2^k
3^\ell}=2k+3\ell$ whenever $k$ and $\ell$ are not both zero.  An
implementation of these algorithms in Haskell is available.
\end{abstract}

\maketitle

\section{Introduction}
\label{intro}

The \emph{complexity} of a natural number $n$ is the least number of $1$'s
needed to write it using any combination of addition and multiplication, with
the order of the operations specified using  parentheses grouped in any legal
nesting.  For instance, $n=11$ has a complexity of $8$, since it can be written
using $8$ ones as
\[ 11=(1+1+1)(1+1+1)+1+1,\] but not with any fewer than $8$.  This notion was
implicitly introduced in 1953 by Kurt Mahler and Jan Popken \cite{MP}; they
actually considered an inverse function, the size of the largest number
representable using $k$ copies of the number $1$.  (More generally, they
considered the same question for representations using $k$ copies of a positive
real number $x$.) Integer complexity was explicitly studied by John Selfridge,
and was later popularized by Richard Guy \cite{Guy, UPINT}.  Following J. Arias
de Reyna \cite{Arias} we will denote the complexity of $n$ by $\cpx{n}$.

Integer complexity is approximately logarithmic; it satisfies the bounds
\begin{equation*}\label{eq1}
3 \log_3 n= \frac{3}{\log 3} \log  n\le \cpx{n} \le \frac{3}{\log 2} \log n  ,\qquad n>1.
\end{equation*}
The lower bound can be deduced from the result of Mahler and Popken, and was
explicitly proved by John Selfridge \cite{Guy}. It is attained with equality for
$n=3^k$ for all $k \ge1$.  The upper bound can be obtained by writing $n$ in
binary and finding a representation using Horner's algorithm. It is not sharp,
and the constant $\frac{3}{\log2} $ can be improved for large $n$ \cite{upbds}.

One can compute $\cpx{n}$ via dynamic programming, since $\cpx{1}=1$, and for
$n>1$, one has
\[
\cpx{n}=\min_{\substack{a,b<n\in \mathbb{N} \\ a+b=n\ \mathrm{or}\ ab=n}}
	(\cpx{a}+\cpx{b}).
\]
This yields an algorithm for computing $\cpx{n}$ that runs in time
$\Theta(n^2)$; in the multiplication case, one needs to check $a\le \sqrt{n}$,
and, na\"ively, in the addition case, one needs to check $a\le n/2$.  However,
Srinivas and Shankar \cite{waset} showed that the upper bound on the addition
case be improved, lowering the time required to $O(n^{\log_2 3})$, by
taking advantage of the inequality $\cpx{n}\ge 3\log_3 n$ to rule out cases when
$a$ is too large.  Arias de Reyna and Van de Lune \cite{anv} took this further
and showed that it could be computed in time $O(n^\alpha)$, where
\[ \alpha=
\frac{\log(3^6 2^{-10}(30557189+21079056\sqrt[3]{3}+14571397\sqrt[3]{9}))}
{\log(2^{10}3^7)}
<1.231;\]
this remains the best known algorithm for computing $\cpx{n}$ for general $n$.

The notion of integer complexity is similar in spirit but different in detail
from the better known measure of \emph{addition chain length}, which has
application to computation of powers, and which is discussed in detail in
Knuth \cite[Sect. 4.6.3]{TAOCP2}.  See also \cite{adcwo} for some interesting
analogies between them; we will discuss this further in
Section~\ref{fksec}.

\subsection{Stability considerations}

One of the easiest cases of complexity to determine is powers of $3$; for any
$k\ge1$, one has $\cpx{3^k}=3k$.  It's clear that $\cpx{3^k}\le 3k$ for any
$k\ge1$, and the reverse inequality follows from Equation~\eqref{eq1}.

The fact that $\cpx{3^k}=3k$ holds for all $k\ge 1$ might prompt one to ask
whether in general it is true that $\cpx{3n}=\cpx{n}+3$.  This is false for
$n=1$, but it does not seem an unreasonable guess for $n>1$.  Nonetheless, this
does not hold; the next smallest counterexample is $n=107$, where $\cpx{107}=16$
but $\cpx{321}=18$.  Indeed, not only do there exist $n$ for which
$\cpx{3n}<\cpx{3n}+3$, there are $n$ for which $\cpx{3n}<\cpx{n}$; one
example is $n=4721323$.  Still, this guess can be rescued.  Let us first make a
definition:

\begin{defn}
A number $m$ is called \emph{stable} if $\cpx{3^k m}=3k+\cpx{m}$ holds for every
$k \ge 0$.
Otherwise it is called \emph{unstable}.
\end{defn}

In \cite{paper1}, this author and Zelinsky showed:
\begin{thm}
\label{basicstab}
For any natural number $n$, there exists $K\ge 0$ such that $3^K n$ is stable.
That is to say, there exists a minimal $K:= K(n)$ such that for any $k\ge K$,
\[ \cpx{3^k n}=3(k-K)+\cpx{3^K n}. \]
\end{thm}

This can be seen as a ``rescue'' of the incorrect guess that
$\cpx{3n}=\cpx{n}+3$ always.  With this theorem, it makes sense to define:

\begin{defn}
Given $n\in \mathbb{N}$, define $K(n)$, the \emph{stabilization length} of $n$,
to be the smallest $k$ such that $3^k n$ is stable.
\end{defn}

We can also define the notion of the \emph{stable complexity} of $n$ (see
\cite{paperwo}), which is, intuitively, what the complexity of $n$ would be ``if
$n$ were stable'':

\begin{defn}
For a positive integer $n$, we define the \emph{stable complexity} of $n$,
denoted $\cpx{n}_\st$, to be $\cpx{3^k n}-3k$ for any $k$ such that $3^k n$ is
stable.  This is well-defined; if $3^k n$ and $3^\ell n$ are both stable, say
with $k\le \ell$, then
\[\cpx{3^k n}-3k=3(k-\ell)+\cpx{3^\ell n}-3k=\cpx{3^\ell n}-3\ell.\]
\end{defn}

The paper \cite{paper1}, while proving the existence of $K(n)$,
gave no upper bound on $K(n)$ or indeed any way of computing it.  Certainly one
cannot compute whether or not $n$ is stable simply by computing for all $k$ the
complexity of $3^k n$; one can guarantee that $n$ is unstable by such
computations, but never that it is stable.  And it's not clear that
$\cpx{n}_\st$, though it has been a useful object of study in \cite{paperwo},
can actually be computed.

\subsection{Main result}
We state the main result.

\begin{thm}
\label{frontpagethm}
We have:
\begin{enumerate}
\item The function $K(n)$, the stabilization length of $n$, is a computable
function of $n$.
\item The function $\cpx{n}_\st$, the stable complexity of $n$, is a computable
function of $n$.
\item The set of stable numbers is a computable set.
\end{enumerate}
\end{thm}

It's worth observing here that, strictly speaking, all three parts of this
theorem are equivalent.  If one has an algorithm for computing $K(n)$, then one
may check whether $n$ is stable by checking whether $K(n)=0$, and one may
compute $\cpx{n}_\st$ by computing $\cpx{3^{K(n)} n}$ by the usual methods and
observing that
\[ \cpx{n}_\st = \cpx{3^{K(n)} n}-3K(n).\]
Similarly, if one has an algorithm for computing $\cpx{n}_\st$, one may compute
whether $n$ is stable by checking if $\cpx{n}_\st = \cpx{n}$.  Finally, if one
has an algorithm for telling if $n$ is stable, one may determine $K(n)$ by
simply applying this algorithm to $n$, $3n$, $9n$, \ldots, until it returns a
positive result, which must eventually occur.  Such methods for converting
between $K(n)$ and $\cpx{n}_\st$ may be quite slow, however.  Fortunately, the
algorithm described here (Algorithm~\ref{stabalg}) will yield both $K(n)$ and
$\cpx{n}_\st$ at once, averting such issues; and if one has $K(n)$, checking
whether $n$ is stable is a one-step process.

\subsection{Applications}

An obvious question about $\cpx{n}$ is that of the complexity of powers,
generalizing what was said about powers of $3$ above.  Certainly for $k\ge 1$ it
is true that
\[ \cpx{n^k} \le k \cpx{n}, \]
and as noted earlier in the case $n=3$ we have equality.  However other values
of $n$ have a more complicated behavior.  For instance, powers of $5$ do not
work nicely, as $\cpx{5^6}=29< 30= 6\cdot\cpx{5}$. The behavior of powers of $2$
remains unknown; it has previously been verified \cite{data2} that
\[
\cpx{2^k} = k \cpx{2} = 2k  ~~\mbox{for} ~~ 1\le k \le 39.
\] 
One may combine the known fact that $\cpx{3^k}=3k$ for $k\ge 1$, and the hope
that $\cpx{2^k}=2k$ for $k\ge 1$, into the following conjecture:

\begin{conj}
For $k,\ell\ge 0$ and not both equal to $0$,
\[ \cpx{2^k 3^\ell}=2k+3\ell. \]
\end{conj}

Such a conjecture, if true, is quite far from being proven; after all, it would
require that $\cpx{2^k}=2k$ for all $k\ge 1$, which would in turn imply that
\[ \limsup_{n\to\infty} \frac{\cpx{n}}{\log n}\ge \frac{2}{\log 2}; \]
at present, it is not even known that this limit is any greater than
$\frac{3}{\log 3}$, i.e., that $\cpx{n}\nsim3\log_3 n$.  Indeed, some have
suggested that $\cpx{n}$ may indeed just be asymptotic to $3\log_3 n$; see
\cite{Guy}.

Nonetheless, in this paper we provide some more evidence for this conjecture, by
proving:

\begin{thm}
\label{frontpage2comput}
For $k\le \NUM$ and arbitrary $\ell$, so long as $k$ and $\ell$ are not both
zero,
\[ \cpx{2^k 3^\ell} = 2k+3\ell.\]
\end{thm}

This extends the results of \cite{data2} regarding numbers of the form $2^k
3^\ell$, as well as the results of \cite{paper1}, which showed this for $k\le
21$ and arbitrary $\ell$.  We prove this not by careful hand analysis, as was
done in \cite{paper1},
but by demonstrating, based on the methods of \cite{theory}, a new algorithm
(Algorithm~\ref{pow2alg}) for computing $\cpx{2^k}$.  Not only does it runs much
faster than existing algorithms, but it also works, as discussed above, by
determining $\cpx{2^k}_\st$ and $K(2^k)$, thus telling us
whether or not, for the given $k$, $\cpx{2^k 3^\ell}=2k+3\ell$ holds for all
$\ell\ge 0$.

The algorithms here can be used for more purposes as well; see
Theorem~\ref{smallunstab} for a further application of them.

\section{Summary of internals and further discussion}
\label{intro2}

\subsection{The defect, low-defect polynomials, and truncation}
\label{introdetail}

Let us now turn our attention to the inner workings of these algorithms, which
are based on the methods in \cite{theory}.
Proving the statement $\cpx{n}=k$ has two parts; showing that $\cpx{n}\le k$,
and showing that $\cpx{n}\ge k$.  The former is, comparatively, the easy part,
as it consists of just finding an expression for $n$ that uses at most $k$ ones;
the latter requires ruling out shorter expressions.  The simplest method for
this is simply exhaustive search, which, as has been mentioned, takes time
$\Theta(n^2)$, or time $O(n^{1.24625})$ once some possibilities have been
eliminated from the addition case.

In this paper, we take a different approach to lower bounding the quantity
$\cpx{n}$, one used earlier in the paper \cite{paper1}; however, we make a
number of improvements to the method of \cite{paper1} that both turn this method
into an actual algorithm, and frequently allow it to run in a reasonable time.
The method is based on considering the \emph{defect} of $n$:

\begin{defn}
The \emph{defect} of $n$, denoted $\dft(n)$ is defined by
\[ \dft(n) := \cpx{n} - 3\log_3 n. \]
\end{defn}

Let us further define:

\begin{defn}
For a real number $s\ge0$, the set $A_s$ is the set of all natural numbers with
defect less than $s$.
\end{defn}

The papers \cite{paperwo,theory,paper1} provided a method of, for any choice of
$\alpha\in(0,1)$, recursively building up descriptions of the sets $A_\alpha,
A_{2\alpha}, A_{3\alpha}, \ldots$; then, if for some $n$ and $k$ we can use this
to demonstrate that $n\notin A_{k\alpha}$, then we have determined a lower bound
on $\cpx{n}$.
More precisely, they showed
that for any $s\ge0$, there is a finite set $\sT_s$ of multilinear polynomials,
of a particular form called \emph{low-defect polynomials}, such that $\dft(n)<s$
if and only if $n$ can be written as $f(3^{k_1},\ldots,3^{k_r})3^{k_{r+1}}$ for
some $f\in \sT_s$ and some $k_1,\ldots,k_{r+1}\ge 0$.  In this paper, we take
this method and show how the polynomials can be produced by an actual algorithm,
and how further useful information can be computed once one has these
polynomials.

In brief, the algorithm works as follows: First, we choose a step size
$\alpha\in(0,1)$.
We start with a set of low-defect polynomials representing $A_\alpha$, and apply
the method of \cite{paperwo} to build up sets representing $A_{2\alpha},
A_{3\alpha}, \ldots$; at each step, we use ``truncation'' method of
\cite{theory} to ensure we are
representing the set $A_{i\alpha}$ exactly and not including extraneous
elements.  Then we check whether or not $n\in A_{i\alpha}$; if it is not, we
continue on to $A_{(i+1)\alpha}$.  If it is, then we have a representation
$n=f(3^{k_1},\ldots,3^{k_r})3^{k_{r+1}}$, and this gives us an upper bound on
$\cpx{n}$; indeed, we can find a shortest
representation for $n$ in this way, and so it gives us $\cpx{n}$ exactly.

This is, strictly speaking, a little different than what was described above, in
that it does not involve directly getting a lower bound on $\cpx{n}$ from the
fact that $n\notin A_{i\alpha}$.  However, this can be used too, so long as we
know in advance an upper bound on $\cpx{n}$.  For instance, this is quite useful
when $n=2^k$ (for $k\ge 1$), as then we know that $\cpx{n}\le 2k$, and hence
that $\dft(n)\le k\dft(2)$.  So we can use the method of the above paragraph,
but stop early, once we have covered defects up to $k\dft(2)-1$.  If we get a
hit within that time, then we have found a shortest representation for $n=2^k$.
Conversely, if $n$ is not detected, then we know that we must have
\[ \dft(2^k)>k\dft(2)-1, \]
and hence that
\[ \cpx{2^k}>2k-1,\]
i.e., $\cpx{2^k}=2k$, thus verifying that the obvious representation is the best
possible.  Again, though we have illustrated it here with powers of $2$, this
method can be used whenever we know in advance an upper bound on $\cpx{n}$; see
Appendix~\ref{impnotes}.

Now, so far we've discussed using these methods to compute $\cpx{n}$, but we
can go further and use them to prove Theorem~\ref{frontpagethm}, i.e., use them
to compute $K(n)$ and $\cpx{n}_\st$.  In this case, at each step, instead of
checking whether there is some $f\in \sT_{i\alpha}$ such that
$n=f(3^{k_1},\ldots,3^{k_r})3^{k_{r+1}}$, we check whether is some $f\in
\sT_{i\alpha}$ and some $\ell$ such that
\[ 3^\ell n = f(3^{k_1},\ldots,3^{k_r})3^{k_{r+1}}. \]
It is not immediately obvious that this is possible, since na\"ively we would
need to check infinitely many $\ell$, but Lemma~\ref{maxv3} allows us to do
this while checking only finitely many $\ell$.  Once we have such a detection,
we can use the value of $\ell$ to determine $K(n)$, and the representation of
$3^\ell n$ obtained this way to determine $\cpx{3^{K(n)} n}$ and hence
$\cpx{n}_\st$.  In addition, if we know in advance an upper bound on $\cpx{n}$,
we can use the same trick as above to sometimes cut the computation short and
conclude not only that $\cpx{n}=k$ but also that $n$ is stable.

\subsection{Comparison to addition chains}
\label{fksec}

It is worth discussing some work analogous to this paper in the study of
addition chains.  An \emph{addition chain} for $n$ is defined to be a sequence
$(a_0,a_1,\ldots,a_r)$ such that $a_0=1$, $a_r=n$, and, for any $1\le k\le r$,
there exist $0\le i, j<k$ such that $a_k = a_i + a_j$; the number $r$ is called
the length of the addition chain.  The shortest length among addition chains for
$n$, called the \emph{addition chain length} of $n$, is denoted $\acl(n)$.
Addition chains were introduced in 1894 by H.~Dellac \cite{Dellac} and
reintroduced in 1937 by A.~Scholz \cite{aufgaben}; extensive surveys on the
topic can be found in Knuth \cite[Section 4.6.3]{TAOCP2} and Subbarao
\cite{subreview}.

The notion of addition chain length has obvious similarities to that of integer
complexity; each is a measure of the resources required to build up the number
$n$ starting from $1$.  Both allow the use of addition, but integer complexity
supplements this by allowing the use of multiplication, while addition chain
length supplements this by allowing the reuse of any number at no additional
cost once it has been constructed.  Furthermore, both measures are approximately
logarithmic; the function $\acl(n)$ satisfies
\[ \log_2 n \le \acl(n) \le 2\log_2 n. \]
 
A difference worth noting is that $\acl(n)$ is actually known to be asymptotic
to $\log_2 n$, as was proved by Brauer \cite{Brauer}, but the function $\cpx{n}$
is not known to be asymptotic to $3\log_3 n$; the value of the quantity
$\limsup_{n\to\infty} \frac{\cpx{n}}{\log n}$ remains unknown.  As mentioned
above, Guy \cite{Guy} has asked whether $\cpx{2^k}=2k$ for $k\ge 1$; if true, it
would make this quantity at least $\frac{2}{\log 2}$.  The Experimental
Mathematics group at the University of Latvia
\cite{data2} has checked that this is true for $k\le 39$.

Another difference worth noting is that unlike integer complexity, there is no
known way to compute addition chain length via dynamic programming.
Specifically, to compute integer complexity this way, one may use the fact that
for any $n>1$,

\begin{displaymath}
\cpx{n}=\min_{\substack{a,b<n\in \mathbb{N} \\ a+b=n\ \mathrm{or}\ ab=n}}
	(\cpx{a}+\cpx{b}).
\end{displaymath}

By contrast, addition chain length seems to be harder to compute.  Suppose we
have a shortest addition chain $(a_0,\ldots,a_{r-1},a_r)$ for $n$; one might
hope that $(a_0,\ldots,a_{r-1})$ is a shortest addition chain for $a_{r-1}$, but
this need not be the case.  An example is provided by the addition chain
$(1,2,3,4,7)$; this is a shortest addition chain for $7$, but $(1,2,3,4)$ is not
a shortest addition chain for $4$, as $(1,2,4)$ is shorter.   Moreover, there is
no way to assign to each natural number $n$ a shortest addition chain
$(a_0,\ldots,a_r)$ for $n$ such that $(a_0,\ldots,a_{r-1})$ is the addition
chain assigned to $a_{r-1}$ \cite{TAOCP2}. This can be an obstacle both to
computing addition chain length and proving statements about addition chains.

Nevertheless, the algorithms described here seem to have a partial analogue for addition
chains in the work of A.~Flammenkamp \cite{fk}.  We might define the
\emph{addition chain defect} of $n$ by
\[ \dft^{\acl}(n):= \acl(n)-\log_2 n; \]
a closely related quantity, the number of \emph{small steps} of $n$, was
introduced by Knuth \cite{TAOCP2}.  The number of small steps of $n$ is defined
by \[ s(n) := \acl(n) - \lfloor \log_2 n \rfloor;\] clearly, this is related to
$\dft^{\acl}(n)$ by $s(n)=\lceil\dft^{\acl}(n)\rceil$.

In 1991, A.~Flammenkamp determined a method for producing descriptions of all
numbers $n$ with $s(n)\le k$ for a given integer $k$, and produced such
descriptions for $k\le 3$ \cite{fk}.  Note that for $k$ an integer, $s(n)\le k$
if and only if $\dft^{\acl}(n)\le k$, so this is the same as determining all $n$
with $\dft^{\acl}(n)\le k$, restricted to the case where $k$ is an integer.
Part of what Flammenkamp proved may be summarized as the following:

\begin{thm}[Flammenkamp]
For any integer $k\ge 0$, there exists a finite set $\sS_k$ of polynomials (in
any number of variables, with nonnegative integer coefficients) such that for
any $n$, one has $s(n)\le k$ if and only if one can write
$n=f(2^{m_1},\ldots,2^{m_r})2^{m_{r+1}}$ for some $f\in \sS_k$ and some integers
$m_1,\ldots,m_{r+1}\ge0$.  Moreover, $\sS_k$ can be effectively computed.
\end{thm}

Unfortunately, the polynomials used in Flammenkamp's method are more complicated
than those produced by the algorithms here; for instance, they cannot always be
taken to be multilinear.  Nonetheless, there is a distinct similarity.

Flammenkamp did not consider questions of stability (which in this case would
result from repeated multiplication by $2$ rather than by $3$; see \cite{adcwo}
for more on this), but it may be possible to use his methods to compute
stability information about addition chains, just as the algorithms here may be
used to compute stability information about integer complexity.  The problem of
extending Flammenkamp's methods to allow for non-integer cutoffs seems more
difficult.

\subsection{Discussion: Algorithms}
\label{discussion1}

Many of the algorithms described here are parametric, in that they require a
choice of a ``step size'' $\alpha\in(0,1)$.  In the attached implementation,
$\alpha$ is always taken to be $\dft(2)=0.107\ldots$, and some precomputations
have been made based on this choice.  See Appendix~\ref{impnotes} for more on
this.  Below, when we discuss the computational complexity of the algorithms
given here, we are assuming a fixed choice of $\alpha$.  It is possible that the
value of $\alpha$ affects the time complexity of these algorithms.  One could
also consider what happens when $\alpha$ is considered as an input to the
algorithm, so that one cannot do pre-computations based on the choice of
$\alpha$.  (In this case we should really restrict the form of $\alpha$ so that
the question makes sense, for instance to $\alpha=p-q\log_3 n$, with $n$ a
natural number and $p,q\in\Q$.)  We will avoid these issues for now, and
assume for the rest of this section that $\alpha=\dft(2)$ unless otherwise
specified.  Two of the algorithms here optionally allow a second input, a known
upper bound $L$ on $\cpx{n}$.  If no bound is input, we may think of this as
$L=\infty$.  We will assume here the simplest case, where no bound $L$ is input,
or equivalently where we always pick $L=\infty$.

We will not actually conduct here a formal analysis of the time
complexity of Algorithm~\ref{stabalg} or Algorithm~\ref{pow2alg}.  Our assertion
that Algorithm~\ref{pow2alg} is much faster than existing methods for computing
$\cpx{2^k}$ is an empirical one.  The speedup is a dramatic one, though; for
instance, the Experimental Mathematics group's computation of $\cpx{n}$ for
$n\le 10^{12}$ required about
3 weeks on a supercomputer, although they used the $\Theta(n^2)$-time algorithm
rather than any of the improvements \cite{letter}; whereas computing
$\cpx{2^\NUM}$ via Algorithm~\ref{pow2alg} required only around 20 hours on the
author's laptop computer.

Empirically, increasing $k$ by one seems to
approximately double the run time of Algorithm~\ref{pow2alg}.  This suggests
that perhaps Algorithm~\ref{pow2alg} runs in time $O(2^k)$, which would be
better than the $O(2^{1.231k})$ bound coming from applying existing methods
\cite{anv} to compute the complexity of $\cpx{2^k}$.

For Algorithm~\ref{stabalg}, the run time seems to be determined more by the
size of $\dft_\st(n):=\cpx{n}_\st-3\log_3 n$ (or by the size of $\dft(n)$, in
the case of Algorithm~\ref{genalg}), rather than by the size of $n$, since it
seems that most of the work consists of building the
sets of low-defect polynomials, rather than checking if $n$ is represented.  For
this reason, computing $\cpx{n}$ via Algorithm~\ref{genalg} is frequently much
slower than using existing methods, even though it is much faster for powers of
$2$.  Note that strictly speaking, $\dft(n)$ can be bounded in terms of $n$,
since
\[ \dft(n) \le 3\log_2 n - 3\log_3 n, \]
but as mentioned earlier, this may be a substantial overestimate.  So it is
worth asking the question:

\begin{quest}
What is the time complexity of Algorithm~\ref{pow2alg}, for computing $K(2^k)$
and $\cpx{2^k}_\st$?  What is the time complexity of of Algorithm~\ref{stabalg}
(with $L=\infty)$, for
computing $K(n)$ and $\cpx{n}_\st$?  What is the time complexity of
Algorithm~\ref{genalg} (with
$L=\infty$), for computing the values of $\cpx{3^k n}$ for a given $n$ and all
$k\ge0$?  What if $L$ may be finite?  How do these depend on the parameter
$\alpha$?  What if $\alpha$ is an input?
\end{quest}

\subsection{Discussion: Stability and computation}
\label{discussion2}

Although we have now given a means to compute
$K(n)$, we have not provided any explicit upper bound on it.  The same is true
for the quantity
\[ \drop(n):=\cpx{n}-\cpx{n}_\st,\]
which is another way of measuring ``how unstable'' the number $n$ is, and which
is also now computable due to Theorem~\ref{frontpagethm}.  We also do not have
any
reliable method of generating unstable numbers with which to demonstrate lower
bounds.

Empirically, large instabilities -- measured either by $K(n)$ or by
$\drop(n)$ -- seem to be rare.  This statement is not based on
running Algorithm~\ref{stabalg} on many numbers to determine their stability, as
that is quite slow in general, but rather on simply computing $\cpx{n}$ for
$n\le 3^{15}$ and then checking $\cpx{n}$, $\cpx{3n}$, $\cpx{9n}$,\ldots, and
guessing that $n$ is stable if no instability is detected before the data runs
out, a method that can only ever put lower bounds on $K(n)$ and $\drop(n)$,
never upper bounds.  Still, numbers that are detectably unstable at all seem to
be somewhat rare, although they still seem to make up a positive fraction of all
natural numbers; namely, around $3\%$.  Numbers that are more than merely
unstable -- having $K(n)\ge 2$ or $\drop(n)\ge 2$ -- are rarer.

The largest lower bounds on $K(n)$ or $\drop(n)$ for a given $n$ encountered
based on these computations are $n=4721323$, which, as mentioned earlier, has
$\cpx{3n}<\cpx{n}$ and thus $\drop(n)\ge 4$; and $17$ numbers, the smallest of
which is $n=3643$, which have $\cpx{3^5 n}<\cpx{3^4 n}+3$ and thus $K(n)\ge 5$.
Finding $n$ where both $K(n)$ and $\drop(n)$ are decently large is hard; for
instance, these computations did not turn up any $n$ for which it could be seen
that both $K(n)\ge 3$ and $\drop(n)\ge 3$.  (See Table~\ref{droptable} for
more.)  It's not even clear whether $K(n)$ or $\drop(n)$ can get arbitrarily
large, or are bounded by some finite constant, although there's no clear reason
why the latter would be so.  Still, this is worth pointing out as a question:

\begin{quest}
What is the natural density of the set of unstable numbers?  What is an explicit
upper bound on $K(n)$, or on $\drop(n)$?  Can $K(n)$ and $\drop(n)$ get
arbitrarily large, or are they bounded?
\end{quest}

Further questions along these lines suggest themselves, but these questions seem
difficult enough, so we will stop this line of inquiry there for now.

\begin{table}
\caption{Numbers that seem to have unusual drop patterns.  Here, the ``drop
pattern'' of $n$ is the list of values $\dft(3^k n)-\dft(3^{k+1} n)$, or
equivalently $\cpx{3^k n}-\cpx{3^{k+1} n}+3$, up until
the point where this is always zero.  This table is empirical, based on a
computation of $\cpx{n}$ for $n\le 3^{15}$; it's possible these numbers have
later drops further on.  Numbers which are divisible by $3$ are not listed.}
\begin{tabular}{l|l}
\label{droptable}
Drop pattern & Numbers with this pattern \\
\hline
$4$ &  $4721323$ \\
$1,2$ & $1081079$ \\
$2,1$ & $203999$, $1328219$ \\
$1,0,0,1$ & $153071, 169199$ \\
\end{tabular}
\end{table}

Strictly speaking, it is possible to prove
Theorem~\ref{frontpagethm} using algorithms based purely on the methods of
\cite{paperwo}, without actually using the ``truncation'' method of the paper
\cite{theory}.  Of course, one cannot simply remove the truncation step from the
algorithms here and get correct answers; other checks are necessary to
compensate.  See Appendix~\ref{impnotes} for a brief discussion of this.
However, while this is sufficient to prove Theorem~\ref{frontpagethm}, the
algorithms obtained this way are simply too slow to be of any use.  And
without the method of truncation, one cannot
write Algorithm~\ref{mainalg},
without which proving Theorem~\ref{smallunstab} would be quite difficult.
We will demonstrate further applications of
the Theorem~\ref{smallunstab} and the method of truncation in future papers
\cite{seq3, stab}.

We can also ask about the computational complexity of computing these functions
in general, rather than just the specific algorithms here.  As noted above, the
best known algorithm for computing $\cpx{n}$ takes time $O(n^{1.24625})$.  It
is also known \cite{Arias} that the problem ``Given $n$ and $k$ in binary, is
$\cpx{n}\le k$?'' is in the class $NP$, because the size of a witness is $O(\log
n)$.  (This problem is not known to be $NP$-complete.)  However, it's not clear
whether the problem ``Given $n$ and $k$ in binary, is $\cpx{n}_\st \le k$?'' is
in the class $NP$, because there's no obvious bound on the size of a witness.
It is quite possible that it could be proven to be in $NP$, however, if an
explicit upper bound could be obtained on $K(n)$.

We can also consider the problem of computing the defect ordering, i.e., ``Given
$n_1$ and $n_2$ in binary, is $\dft(n_1)\le \dft(n_2)$?''; the significance of
this problem is that the set of all defects is in fact a well-ordered set
\cite{paperwo} with order type $\omega^\omega$.  This problem lies in
$\Delta_2^P$ in the polynomial hierarchy \cite{paperwo}.  The paper
\cite{paperwo} also defined the \emph{stable defect} of $n$:
\begin{defn}
\label{dftstearlydef}
The \emph{stable defect} of $n$, denoted $\dft_\st(n)$, is
\[ \dft_\st(n) := \cpx{n}_\st - 3\log_3 n. \]
\end{defn}
(We will review the stable defect and its properties in
Section~\ref{subsecdft}.)  Thus we get the problem of, ``Given $n_1$ and $n_2$
in binary, is $\dft_\st(n_1) \le \dft_\st(n_2)$?''  The
image of $\dft_\st$ is also well-ordered with order type $\omega^\omega$, but
until now it was not known that this problem is computable.  But
Theorem~\ref{frontpagethm} shows that it is, and so we can ask about its
complexity.  Again, due to a lack of bounds on $K(n)$, it's not clear that this
lies in $\Delta_2^P$.

We can also ask about the complexity of computing $K(n)$, or $\drop(n)$ (which,
conceivably, could be easier than $\cpx{n}$ or $\cpx{n}_\st$, though this seems
unlikely), or, perhaps most importantly, of computing a set $\sT_s$ for a given
$s\ge0$.  Note that in this last case, it need not be the set $\sT_s$ found by
Algorithm~\ref{mainalg} here; we just want any set satisfying the required
properties -- a \emph{good covering} of $B_s$, as we call it here
(see Definition~\ref{goodcover}).  Of course, we must make a restriction on the
input for this last question, as one cannot actually take arbitrary real numbers
as input; perhaps it would be appropriate to restrict to $s$ of the form
\[ s \in \{ p-q\log_3 n : p,q\in \Q, n\in\N \}, \]
which seems like a large enough set of real numbers to cover all the numbers we
care about here.

We summarize:

\begin{quest}
What is the complexity of computing $\cpx{n}$?  What is the complexity of
computing $\cpx{n}_\st$?  What is the complexity of computing the difference
$\drop(n)$?  What is the complexity of computing the defect ordering
$\dft(n_1)\le \dft(n_2)$?  What is the complexity of computing the stable defect
ordering $\dft_\st(n_1)\le\dft_\st(n_2)$?  What is the complexity of computing
the stabilization length $K(n)$?
\end{quest}

\begin{quest}
Given $s=p-q\log_3 n$, with $p,q\in \Q$ and $n\in\N$, what is the complexity of
computing a good covering $\sT_s$ of $B_s$?
\end{quest}

\section{The defect, stability, and low-defect polynomials}
\label{review}

In this section we will review the results of \cite{paperwo} and \cite{theory}
regarding the defect $\dft(n)$, the stable complexity $\cpx{n}_\st$, and
low-defect polynomials.

\subsection{The defect and stability}
\label{subsecdft}

First, some basic facts about the defect:

\begin{thm}
\label{oldprops}
We have:
\begin{enumerate}
\item For all $n$, $\dft(n)\ge 0$.
\item For $k\ge 0$, $\dft(3^k n)\le \dft(n)$, with equality if and only if
$\cpx{3^k n}=3k+\cpx{n}$.  The difference $\dft(n)-\dft(3^k n)$ is a nonnegative
integer.
\item A number $n$ is stable if and only if for any $k\ge 0$, $\dft(3^k
n)=\dft(n)$.
\item If the difference $\dft(n)-\dft(m)$ is rational, then $n=m3^k$ for some
integer $k$ (and so $\dft(n)-\dft(m)\in\mathbb{Z}$).
\item Given any $n$, there exists $k$ such that $3^k n$ is stable.
\item For a given defect $\alpha$, the set $\{m: \dft(m)=\alpha \}$ has either
the form $\{n3^k : 0\le k\le L\}$ for some $n$ and $L$, or the form $\{n3^k :
0\le k\}$ for some $n$.  This latter occurs if and only if $\alpha$ is the
smallest defect among $\dft(3^k n)$ for $k\in \mathbb{Z}$.
\item $\dft(1)=1$, and for $k\ge 1$, $\dft(3^k)=0$.  No other integers occur as
$\dft(n)$ for any $n$.
\item If $\dft(n)=\dft(m)$ and $n$ is stable, then so is $m$.
\end{enumerate}
\end{thm}

\begin{proof}
Parts (1) through (7), excepting part (3), are just Theorem~2.1 from
\cite{paperwo}.  Part (3) is Proposition~12 from \cite{paper1}, and part (8) is
Proposition~3.1 from \cite{paperwo}.
\end{proof}

The paper \cite{paperwo} also defined the notion of a \emph{stable defect}:

\begin{defn}
We define a \emph{stable defect} to be the defect of a stable number.
\end{defn}

Because of part (9) of Theorem~\ref{oldprops}, this definition makes sense; a
stable defect $\alpha$ is not just one that is the defect of some stable number,
but one for which any $n$ with $\dft(n)=\alpha$ is stable.  Stable defects can
also be characterized by the following proposition from \cite{paperwo}:

\begin{prop}
\label{modz1}
A defect $\alpha$ is stable if and only if it is the smallest defect
$\beta$ such that $\beta\equiv\alpha\pmod{1}$.
\end{prop}

We can also define the \emph{stable defect} of a given number, which we denote
$\dft_\st(n)$.  (We actually already defined this in
Definition~\ref{dftstearlydef}, but let us disregard that for now and give a
different definition; we will see momentarily that they are equivalent.)

\begin{defn}
For a positive integer $n$, define the \emph{stable defect} of $n$, denoted
$\dft_\st(n)$, to be $\dft(3^k n)$ for any $k$ such that $3^k n$ is stable.
(This is well-defined as if $3^k n$ and $3^\ell n$ are stable, then $k\ge \ell$
implies $\dft(3^k n)=\dft(3^\ell n)$, and so does $\ell\ge k$.)
\end{defn}

Note that the statement ``$\alpha$ is a stable defect'', which earlier we were
thinking of as ``$\alpha=\dft(n)$ for some stable $n$'', can also be read as the
equivalent statement ``$\alpha=\dft_\st(n)$ for some $n$''.

We then have the following facts relating the notions of $\cpx{n}$, $\dft(n)$,
$\cpx{n}_\st$, and $\dft_\st(n)$:

\begin{prop}
\label{stoldprops}
We have:
\begin{enumerate}
\item $\dft_\st(n)= \min_{k\ge 0} \dft(3^k n)$
\item $\dft_\st(n)$ is the smallest defect $\alpha$ such that
$\alpha\equiv \dft(n) \pmod{1}$.
\item $\cpx{n}_\st = \min_{k\ge 0} (\cpx{3^k n}-3k)$
\item $\dft_\st(n)=\cpx{n}_\st-3\log_3 n$
\item $\dft_\st(n) \le \dft(n)$, with equality if and only if $n$ is stable.
\item $\cpx{n}_\st \le \cpx{n}$, with equality if and only if $n$ is stable.
\end{enumerate}
\end{prop}

\begin{proof}
These are just Propositions~3.5, 3.7, and 3.8 from \cite{paperwo}.
\end{proof}

\subsection{Low-defect expressions, polynomials, and pairs}
\label{bdsec}

As has been mentioned in Section~\ref{introdetail}, we are going to represent
the set $A_r$ by substituting in powers of $3$ into certain multilinear
polynomials we call \emph{low-defect polynomials}.  Low-defect polynomials come
from particular sorts of expressions we will call \emph{low-defect expressions}.
We will associate with each polynomial or expression a ``base complexity'' to
from a \emph{low-defect pair}.  In this section we will review the properties of
these polynomials and expressions.

First, their definition:

\begin{defns}
A \emph{low defect expression} is defined to be a an expression in positive
integer constants, $+$, $\cdot$, and some number of variables, constructed
according to the following rules:
\begin{enumerate}
\item Any positive integer constant by itself forms a low-defect expression.
\item Given two low-defect expressions using disjoint sets of variables, their
product is a low-defect expression.  If $E_1$ and $E_2$ are low-defect
expressions, we will use $E_1 \otimes E_2$ to denote the low-defect expression
obtained by first relabeling their variables to disjoint and then multiplying
them.
\item Given a low-defect expression $E$, a positive integer constant $c$, and a
variable $x$ not used in $E$, the expression $E\cdot x+c$ is a low-defect
expression.  (We can write $E\otimes x+c$ if we do not know in advance that $x$
is not used in $E$.)
\end{enumerate}

We also define an \emph{augmented low-defect expression} to be an expression of
the form $E\cdot x$, where $E$ is a low-defect expression and $x$ is a variable
not appearing in $E$.  If $E$ is a low-defect expression, we also use $\xpdd{E}$
to denote the low-defect $E\otimes x$.
\end{defns}

Note that we do not really care about what variables a low-defect expression is
in -- if we permute the variables of a low-defect polynomial or replace them
with others, we will regard the result as an equivalent low-defect expression.

We also define the complexity of a low-defect expression:

\begin{defns}
The \emph{complexity} of a low-defect expression $E$, denoted $\cpx{E}$, is the
sum of the complexities of all the constants used in $E$.  A \emph{low-defect
[expression] pair} is an ordered pair $(E,k)$ where $E$ is a low-defect
expression, and $k$ is a whole number with $k\ge \cpx{E}$.
\end{defns}

The reason for introducing the notion of a ``low-defect pair'' is that we may
not always know the complexity of a given low-defect expression; frequently, we
will only know an upper bound on it.  For more theoretical applications, one
does not always need to keep track of this, but since here we are concerned with
computation, we need to keep track.  One can, of course, always compute the
complexity of any low-defect expression one is given; but to do so may be
computationally expensive, and it is easier to simply keep track of an upper
bound.  (Indeed, for certain applications, one may actually want to keep track
of more detailed information, such as an upper bound on each constant
individually; see Appendix~\ref{impnotes} for more on this.)

One can then evaluate these expressions to get polynomials:

\begin{defns}
A \emph{low-defect polynomial} is a polynomial $f$ obtained by evaluating a
low-defect expression $E$.  If $(E,k)$ is a low-defect [expression] pair, we say
$(f,k)$ is a low-defect [polynomial] pair.  We use $\xpdd{f}$ to refer to the
polynomial obtained by evaluating $\xpdd{E}$, and call it an \emph{augmented
low-defect polynomial}.  For convenience, if $(f,k)$ is a low-defect pair, we
may say ``the degree of $(f,k)$'' to refer to the degree of $f$.
\end{defns}

Typically, for practical use, what we want is not either low-defect expressions
or low-defect polynomials.  Low-defect polynomials do not retain enough
information about how they were made.  For instance, in the algorithms below, we
will frequently want to substitute in values for the ``innermost'' variables in
the polynomial; it is shown in \cite{theory} that this is well-defined even if
multiple expressions can give rise to the same polynomial.  However, if all one
has is the polynomial rather than the expression which generated it, determining
which variables are innermost may require substantial computation.

On the other hand, low-defect expressions contain unneeded information; there
is little practical reason to distinguish between, e.g., $2(3x+1)$ and
$(3x+1)\cdot2$, or between $1\cdot(3x+1)$ and $3x+1$, or $2(2(3x+1))$ and
$4(3x+1)$.  A useful practical representation is what \cite{theory} called a
\emph{low-defect tree}:

\begin{defn}
Given a low-defect expression $E$, we define a corresponding \emph{low-defect
tree} $T$, which is a rooted tree where both edges and vertices are labeled with
positive integers.  We build this tree as follows:
\begin{enumerate}
\item If $E$ is a constant $n$, $T$ consists of a single vertex labeled with
$n$.
\item If $E=E'\cdot x + c$, with $T'$ the tree for $E$, $T$ consists of $T'$
with a new root attached to the root of $T'$.  The new root is labeled with a
$1$, and the new edge is labeled with $c$.
\item If $E=E_1 \cdot E_2$, with $T_1$ and $T_2$ the trees for $E_1$ and $E_2$
respectively, we construct $E$ by ``merging'' the roots of $E_1$ and $E_2$ --
that is to say, we remove the roots of $E_1$ and $E_2$ and add a new root, with
edges to all the vertices adjacent to either of the old roots; the new edge
labels are equal to the old edge labels.  The label of the new root is equal
to the product of the labels of the old roots.
\end{enumerate}
\end{defn}

See Figure~\ref{treeexamp} for an example illustrating this construction.

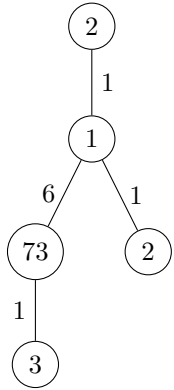
\begin{figure}
\caption{Low-defect tree for the expression $2((73(3x_1+1)x_2+6)(2x_3+1)x_4+1)$.}
\label{treeexamp}
\begin{center}
\begin{tikzpicture}
\node[circle,draw] {$2$}
	child {node[circle,draw] {$1$}
		child { node[circle,draw] {$73$}
			child {node[circle,draw] {$3$}
				edge from parent node[left] {$1$}}
			edge from parent node[left] {$6$}}
		child {node[circle,draw] {$2$}
			edge from parent node[right]{$1$}}
		edge from parent node[right]{$1$}};
\end{tikzpicture}
\end{center}
\end{figure}

This will still contain information that is unnecessary for our purposes -- for
instance, this representation still distinguishes between $4(2x+1)$ and
$2(4x+2)$ -- but it is on the whole a good medium between including too much and
including too little.  While the rest of the paper will discuss low-defect
expressions and low-defect polynomials, we assume these are being represented as
trees, for convenience.

\subsection{Properties of low-defect polynomials}

Having now discussed the definition and representation of low-defect expressions
and polynomials, let us now discuss their properties.

Note first that the degree of a low-defect polynomial is also equal to the
number of variables it uses; see Proposition~\ref{polystruct}.  We will often
refer to the ``degree'' of a low-defect pair $(f,C)$; this refers to the degree
of $f$.  Also note that augmented low-defect polynomials are never low-defect
polynomials; as we will see in a moment (Proposition~\ref{polystruct}),
low-defect polynomials always have nonzero constant term, whereas augmented
low-defect polynomials always have zero constant term.

Low-defect polynomials are multilinear polynomials; indeed, they are read-once
polynomials (in the sense of for instance \cite{ROF}), as low-defect expressions
are easily seen to be read-once expressions.

In \cite{paperwo} were proved the following propositions about low-defect
pairs:

\begin{prop}
\label{polystruct}
Suppose $f$ is a low-defect polynomial of degree $r$.  Then $f$ is a
polynomial in the variables $x_1,\ldots,x_r$, and it is a multilinear
polynomial, i.e., it has degree $1$ in each of its variables.  The coefficients
are non-negative integers.  The constant term is nonzero, and so is the
coefficient of $x_1\ldots x_r$, which we will call the \emph{leading
coefficient} of $f$.
\end{prop}

\begin{prop}
\label{basicub}
If $(f,C)$ is a low-defect pair of degree $r$, then
\[\cpx{f(3^{n_1},\ldots,3^{n_r})}\le C+3(n_1+\ldots+n_r).\]
and
\[\cpx{\xpdd{f}(3^{n_1},\ldots,3^{n_{r+1}})}\le C+3(n_1+\ldots+n_{r+1}).\]
\end{prop}

\begin{proof}
This is a combination of Proposition~4.5 and Corollary~4.12 from \cite{paperwo}.
\end{proof}

Because of this, it makes sense to define:

\begin{defn}
Given a low-defect pair $(f,C)$ (say of degree $r$) and a number $N$, we will
say that $(f,C)$ \emph{efficiently $3$-represents} $N$ if there exist
nonnegative integers $n_1,\ldots,n_r$ such that
\[N=f(3^{n_1},\ldots,3^{n_r})\ \textrm{and}\ \cpx{N}=C+3(n_1+\ldots+n_r).\]
We will say $(\xpdd{f},C)$ efficiently
$3$-represents $N$ if there exist $n_1,\ldots,n_{r+1}$ such that
\[N=\xpdd{f}(3^{n_1},\ldots,3^{n_{r+1}})\ \textrm{and}\ 
\cpx{N}=C+3(n_1+\ldots+n_{r+1}).\]
More generally, we will also say $f$ $3$-represents $N$ if there exist
nonnegative integers $n_1,\ldots,n_r$ such that $N=f(3^{n_1},\ldots,3^{n_r})$.
and similarly with $\xpdd{f}$.
We will also use the same terminology regarding low-defect expressions.
\end{defn}

 Note that if $E$ is a low-defect expression and $(E,C)$ (or $(\xpdd{E},C)$)
efficiently $3$-represents some $N$, then $(E,\cpx{E})$ (respectively,
$(\xpdd{E},\cpx{E})$ efficiently $3$-represents $N$, which means that in order
for $(E,C)$ (or $(\xpdd{E},C)$ to $3$-represent anything efficiently at all, we
must have $C=\cpx{E}$.  And if $f$ is a low-defect polynomial and $(f,C)$ (or
$\xpdd{f},C)$ efficiently $3$-represents some $N$, then $C$ must be equal to the
smallest $\cpx{E}$ among any low-defect expression $E$ that evaluates to $f$
(which in \cite{paperwo} and \cite{theory} was denoted $\cpx{f}$).  But, again,
it is still worth using low-defect pairs rather than just low-defect polynomials
and expressions since we do not want to spend time computing the value
$\cpx{E}$.

For this reason it makes sense to use ``$E$ efficiently $3$-represents $N$'' to
mean ``some $(E,C)$ efficiently $3$-represents $N$'' or equivalently
``$(E,\cpx{E})$ efficiently $3$-reperesents $N$''.  Similarly with $\xpdd{E}$.

In keeping with the name, numbers $3$-represented by low-defect polynomials, or
their augmented versions, have bounded defect.  Let us make some definitions
first:

\begin{defn}
Given a low-defect pair $(f,C)$, we define $\dft(f,C)$, the defect of $(f,C)$,
to be $C-3\log_3 a$, where $a$ is the leading coefficient of $f$.
\end{defn}

\begin{defn}
Given a low-defect pair $(f,C)$ of degree $r$, we define
\[\dft_{f,C}(n_1,\ldots,n_r) =
C+3(n_1+\ldots+n_r)-3\log_3 f(3^{n_1},\ldots,3^{n_r}).\]
\end{defn}

Then we have:

\begin{prop}
\label{dftbd}
Let $(f,C)$ be a low-defect pair of degree $r$, and let $n_1,\ldots,n_{r+1}$ be
nonnegative integers.
\begin{enumerate}
\item We have
\[ \dft(\xpdd{f}(3^{n_1},\ldots,3^{n_{r+1}}))\le \dft_{f,C}(n_1,\ldots,n_r)\]
and the difference is an integer.
\item We have \[\dft_{f,C}(n_1,\ldots,n_r)\le\dft(f,C)\]
and if $r\ge 1$, this inequality is strict.
\item The function $\dft_{f,C}$ is strictly increasing in each variable, and
\[ \dft(f,C) = \sup_{k_1,\ldots,k_r} \dft_{f,C}(k_1,\ldots,k_r).\]
\end{enumerate}
\end{prop}

\begin{proof}
This is a combination of Proposition~4.9 and Corollary~4.14 from \cite{paperwo}
along with Proposition~2.14 from \cite{theory}.
\end{proof}

Indeed, one can make even stronger statements than (3) above.  In \cite{theory},
a partial order is placed on the variables of a low-defect polynomial $f$,
where, for variables $x$ and $y$ in $f$, we say $x\preceq y$ if $x$ appears
``deeper'' in a low-defect expression $f$ for than $y$.  Formally,

\begin{defn}
\label{defnest}
Let $E$ be a low-defect expression.  Let $x$ and $y$ be variables appearing in
$E$.  We say that $x\preceq y$ under the nesting ordering for $E$ if $x$ appears
in the smallest low-defect subexpression of $E$ that contains $y$.  (If $f$ is a
low-defect polynomial, it can be shown that the nesting order is independent of
the low-defect expression used to generate it; see Proposition~3.18 from
\cite{theory}.)
\end{defn}

For instance, if $f=((((2x_1+1)x_2+1)(2x_3+1)x_4+1)x_5+1)(2x_6+1)$, one has
$x_1\prec x_2\prec x_4\prec x_5$ and $x_3\prec x_4\prec x_5$ but no other
relations.  It's then shown \cite[Proposition~4.6]{theory} that (3) above is
true even if only the minimal (i.e., innermost) variables are allowed to
approach infinity.  That is to say:

\begin{prop}
\label{minlimit}
Let $(f,C)$ be a low-defect pair of degree $r$.  Say $x_{i_j}$, for $1\le j\le
s$, are the minimal variables of $f$.  Then
\[ \lim_{k_{i_1},\ldots,k_{i_s}\to\infty}
	\dft_{f,C}(k_1,\ldots,k_r) = \dft(f,C)\]
(where the other $k_i$ remain fixed).
\end{prop}

Note that if we store the actual low-defect expression rather than just the
resulting polynomial, finding the minimal variables is easy.

With this, we have the basic properties of low-defect polynomials.

\subsection{Good coverings}
\label{subsecgc}

Finally, before we begin listing algorithms, let us state precisely what
precisely the algorithms are for.  We will first need the notion of a
\emph{leader}:

\begin{defn}
A natural number $n$ is called a \emph{leader} if it is the smallest number with
a given defect.  By part (6) of Theorem~\ref{oldprops}, this is equivalent to
saying that either $3\nmid n$, or, if $3\mid n$, then $\dft(n)<\dft(n/3)$, i.e.,
$\cpx{n}<3+\cpx{n/3}$.
\end{defn}

Let us also define:

\begin{defn}
For any real $r\ge0$, define the set of {\em $r$-defect numbers} $A_r$ to be 
\[A_r := \{n\in\mathbb{N}:\dft(n)<r\}.\]
Define the set of {\em $r$-defect leaders} $B_r$ to be 
\[
B_r:= \{n \in A_r :~~n~~\mbox{is a leader}\}.
\]
\end{defn}

These sets are related by the following proposition from \cite{paperwo}:

\begin{prop}
\label{arbr}
For every $n\in A_r$, there exists a unique $m\in B_r$ and $k\ge 0$ such that
$n=3^k m$ and $\dft(n)=\dft(m)$; then $\cpx{n}=\cpx{m}+3k$.
\end{prop}

Because of this, if we want to describe the set $A_r$, it suffices to describe
the set $B_r$.

As mentioned earlier, what we want to do is to be able to write every number in
$A_r$ as $f(3^{k_1},\ldots,3^{k_r})3^{k_{r+1}}$ for some low-defect polynomial
$f$ drawn from a finite set depending on $r$.  In fact, we want to be able to
write every number in $B_r$ as $f(3^{k_1},\ldots,3^{k_r})$, with the same
restrictions.  So we define:

\begin{defn}
\label{coverdef}
For $r\ge0$, a finite set $\sS$ of low-defect pairs will be called a
\emph{covering set} for $B_r$ if every $n\in B_r$ can be efficiently
$3$-represented by some pair in $\sS$.  (And hence every $n\in A_r$ can be
efficiently represented by some $(\xpdd{f},C)$ with $(f,C)\in \sS$.)
\end{defn}

Of course, this is not always enough; we don't just want that every number in
$A_r$ can be represented in this way, but also that every number generated this
way is in $A_r$.  So we define:

\begin{defn}
\label{goodcover}
For $r\ge0$, a finite set $\sS$ of low-defect pairs will be called a
\emph{good covering} for $B_r$ if every $n\in B_r$ can be efficiently
$3$-represented by some pair in $\sS$ (and hence every $n\in A_r$ can be
efficiently represented by some $(\xpdd{f},C)$ with $(f,C)\in \sS$); for every
$(f,C)\in\sS$, $\dft(f,C)\le r$, with this being strict if $\deg f=0$.
\end{defn}

With this, it makes sense to state the following theorem from \cite{theory}:
\begin{thm}
\label{mainthm}
For any real number $r\ge 0$, there exists a good covering of $B_r$.
\end{thm}

\begin{proof}
This is Theorem~4.9 from \cite{theory} rewritten in terms of
Definition~\ref{goodcover}.
\end{proof}

Computing good coverings, then, will be one of the primary subjects for the rest
of the paper.

Before we continue with that, however, it is also worth noting here the
following proposition from \cite{theory}:

\begin{prop}
\label{polycpxbd}
Let $(f,C)$ be a low-defect pair of degree $k$, and suppose that $a$ is the
leading coefficient of $f$.  Then $C\ge \cpx{a} + k$.  In particular,
$\dft(f,C) \ge \dft(a) + k \ge k$.
\end{prop}

\begin{proof}
This is a combination of Proposition~3.24 and Corollary~3.25 from \cite{theory}.
\end{proof}

This implies that in any good covering of $B_r$, all polynomials have degree at
most $\lfloor r \rfloor$.

\section{Algorithms: Building up covering sets}

Now let us discuss the ``building-up'' method from \cite{paper1} and
\cite{paperwo} that forms one-half the core of the algorithm.  The second
``filtering-down'' half, truncation, will be described in
Section~\ref{secalgos}.  This section will describe how to compute covering
sets for $B_r$ (see Definition~\ref{coverdef}); the next section will describe
how to turn them into good coverings.

Note however that the results of the above sections and previous papers deal
with real numbers, but real numbers cannot be represented exactly in a computer.
Hence, we will for the rest of this section fix a subset $\cR$ of the real
numbers on which we can do exact computation.  For concreteness, we will define

\begin{defn}
The set $\cR$ is the set of all real numbers of the form $q+r\log_3 n$, where
$q$ and $r$ are rational and $n$ is a natural number.
\end{defn}

This will suffice for our purposes; it contains all the numbers we're working
with here.  However it is worth noting that all these algorithms will work just
as well with a larger set of allowed numbers, so long as it supports all the
required operations.

Note that since the algorithms in both this section and later sections consist,
in some cases, of simply using the methods described in proofs of theorems in
\cite{paperwo} and \cite{theory}, we will, in these cases, not give detailed
proofs of correctness; we will simply direct the reader to the proof of the
corresponding theorem.  We will include proofs of correctness only where we are
not directly following the proof of an earlier theorem.

\subsection{Algorithm \ref{startalg}: Computing $B_{\alpha}$, $0 < \alpha <1$.}

The theorems of \cite{paperwo} that build up covering sets for $B_r$ do so
inductively; they require first picking a step size $\alpha\in(0,1)$ and
then determining covering sets $B_{k\alpha}$ for natural numbers $k$.  So first,
we need a base case -- an algorithm to compute $B_\alpha$.  Fortunately, this is
given by the following theorem from \cite{paper1}:

\begin{thm}
\label{finite}
For every $\alpha$ with $0<\alpha<1$, the set of leaders $B_\alpha$ is a finite
set.  More specifically, the list of $n$ with $\dft(n)<1$ is as follows:
\begin{enumerate}
\item $3^\ell$ for $\ell\ge 1$, of complexity $3\ell$ and defect $0$
\item $2^k 3^\ell$ for $1\le k\le 9$, of complexity $2k+3\ell$ and defect
$k\dft(2)$
\item $5\cdot2^k 3^\ell$ for $k\le 3$, of complexity $5+2k+3\ell$ and defect
$\dft(5)+k\dft(2)$
\item $7\cdot2^k 3^\ell$ for $k\le 2$, of complexity $6+2k+3\ell$ and defect
$\dft(7)+k\dft(2)$
\item $19\cdot3^\ell$ of complexity $9+3\ell$ and defect $\dft(19)$
\item $13\cdot3^\ell$ of complexity $8+3\ell$ and defect $\dft(13)$
\item $(3^k+1)3^\ell$ for $k>0$, of complexity $1+3k+3\ell$ and defect
$1-3\log_3(1+3^{-k})$
\end{enumerate}
\end{thm}

Strictly speaking, we do not necessarily need this theorem to the same extent as
\cite{paperwo} needed it; we only need it if we want to be able to choose step
sizes $\alpha$ with $\alpha$ arbitrarily close to $1$.  In \cite{paperwo}, this
was necessary to keep small the degrees of the polynomials; larger steps
translates into fewer steps, which translates into lower degree.  However, in
Section~\ref{secalgos}, we will introduce algorithms for performing truncation,
as described in \cite{theory}; and with truncation, we can limit the degree
without needing large steps (see Corollary~\ref{polycpxbd}), allowing us to keep
$\alpha$ small if we so choose.  For instance, in the attached implementation,
we always use $\alpha=\dft(2)$.  Nonetheless, one may wish to use larger
$\alpha$, so this proposition is worth noting.

The above theorem can be rephrased as our Algorithm~\ref{startalg}:

\begin{algorithm}[ht]
\caption{Determine the set $B_\alpha$}
\label{startalg}
\begin{algorithmic}
\ENSURE $\alpha \in (0,1)\cap \cR$
\REQUIRE $T= \{ (n,k) : n\in B_\alpha, k=\cpx{n} \}$
\STATE $T \gets \{(3,3)\}$
\STATE Determine largest integer $k$ such that $k\dft(2)<\alpha$ and $k\le 9$
\COMMENT{$k$ may be $0$, in which case the following loop never executes}
\FOR{$i=1$ \TO $k$}
	\STATE $T\gets T \cup \{(2^i,2i)\}$
\ENDFOR
\STATE Determine largest integer $k$ such that $\dft(5)+k\dft(2)<\alpha$ and
$k\le 3$
\COMMENT{$k$ may be negative, in which case the following loop never executes}
\FOR{$i=0$ \TO $k$}
	\STATE $T\gets T \cup \{(5\cdot2^i,5+2i)\}$
\ENDFOR
\STATE Determine largest integer $k$ such that $\dft(7)+k\dft(2)<\alpha$ and
$k\le 2$
\COMMENT{$k$ may be negative, in which case the following loop never executes}
\FOR{$i=0$ \TO $k$}
	\STATE $T\gets T \cup \{(7\cdot2^i,6+2i)\}$
\ENDFOR
\IF{$\alpha>\dft(19)$}
	\STATE $T\gets T\cup \{(19,9)\}$
\ENDIF
\IF{$\alpha>\dft(13)$}
	\STATE $T\gets T\cup \{(13,8)\}$
\ENDIF
\STATE Determine largest integer $k$ for which $1-3\log_3(1+3^{-k})<\alpha$
\COMMENT{$k$ may be $0$, in which case the following loop never executes}
\FOR{$i=1$ \TO $k$}
	\STATE $T\gets T \cup \{(3^i+1,1+3i)\}$
\ENDFOR
\RETURN $T$
\end{algorithmic}
\end{algorithm}

\begin{proof}[Proof of correctness for Algorithm~\ref{startalg}]
The correctness of this algorithm is immediate from Theorem~\ref{finite}.
\end{proof}

\subsection{Algorithm \ref{buildalg}: Computing $B_{(k+1)\alpha}$.}

Now we record Algorithm~\ref{buildalg}, for computing a covering set for
$B_{(k+1)\alpha}$ if we have ones already for $B_\alpha,\ldots,B_{k\alpha}$.
This algorithm is essentially the proof of Theorem~4.10 from \cite{paperwo},
though we have made a slight modification to avoid redundancy.

Algorithm~\ref{buildalg} refers to ``solid numbers'', and to a set $T_\alpha$,
notions taken from \cite{paper1}, which we have not thus far defined, so let us
define those here.

\begin{defns}
We say a number $n$ is \emph{solid} if it cannot be efficiently represented as a
a sum, i.e., there do not exist numbers $a$ and $b$ with $a+b=n$ and
$\cpx{a}+\cpx{b}=n$.  We say a number $n$ is \emph{m-irreducible} if it cannot
be efficiently represented as a product, i.e., there do not exist $a$ and $b$
with $ab=n$ and $\cpx{a}+\cpx{b}=n$.  We define the set $T_\alpha$ to consist of
$1$ together with those m-irreducible numbers $n$ which satisfy
\[ \frac{1}{n-1} > 3^{\frac{1-\alpha}{3}} -1 \]
and do not satisfy $\cpx{n}=\cpx{n-b}+\cpx{b}$ for any solid $b$ with $1<b\le
n/2$.
\end{defns}

\begin{algorithm}[ht]
\caption{Compute a covering set $\sS_{k+1}$ for $B_{(k+1)\alpha}$ from covering
sets $\sS_1,\ldots,\sS_k$ for $B_\alpha , \ldots, B_{k\alpha}$}
\label{buildalg}
\begin{algorithmic}
\REQUIRE $k\in\mathbb{N}$, $\alpha\in (0,1)\cap \cR$, $\sS_i$ a covering set for
$B_{i\alpha}$ for $1\le i\le k$
\ENSURE $\sS_{k+1}$ a covering set for $B_{(k+1)\alpha}$
\FORALL{$i=1$ \TO $k$}
	\STATE $\sS'_i \gets \sS_i \setminus \{(1,1),(3,3)\}$
\ENDFOR
\STATE $\sS_{k+1}\gets \emptyset$
\STATE Compute the set $T_\alpha$, and the complexities of its elements; let $U$
be the set $\{(n,\cpx{n}):n\in T_\alpha\}$ \COMMENT{One may use instead a
superset of $T_\alpha$ if determining $T_\alpha$ exactly takes too long}
\STATE Compute the set $V_{k,\alpha}$, the set of solid numbers $n$ such that
$\cpx{n}<(k+1)\alpha+3\log_3 2$
\COMMENT{Again, one may use a superset}
\IF{$k=1$}
	\STATE $\sS_{k+1}\gets \sS_{k+1} \cup \{(f_1 \otimes f_2 \otimes f_3, C_1 + C_2 + C_3) : (f_\ell,C_\ell)\in \sS'_1\}$
	\STATE $\sS_{k+1}\gets \sS_{k+1} \cup \{(f_1 \otimes f_2, C_1 + C_2) : (f_\ell,C_\ell)\in \sS'_1\}$
\ELSE
	\STATE $\sS_{k+1}\gets \sS_{k+1} \cup \{(f \otimes g, C+D) :
		(f,C)\in \sS'_i, (g,D)\in \sS'_j, i+j=k+2\}$
\ENDIF
\STATE $\sS_{k+1}\gets \sS_{k+1}\cup \{(f\otimes x+b,C+\cpx{b}):(f,C)\in
\sS_{k\alpha}, b\in V_{k,\alpha}\}$
\STATE $\sS_{k+1}\gets \sS_{k+1}\cup \{(g\otimes(f\otimes x+b),C+D+\cpx{b}):
(f,C)\in \sS_{k\alpha}, b\in V_{k,\alpha}, (g,D)\in \sS'_1\}$
\STATE $\sS_{k+1}\gets \sS_{k+1}\cup U$
\STATE $\sS_{k+1}\gets \sS_{k+1}\cup \{(f\otimes g, C+D): f\in U, g\in
\sS'_1\}$
\RETURN $\sS_{k+1}$
\end{algorithmic}
\end{algorithm}

\begin{proof}[Proof of correctness for Algorithm~\ref{buildalg}]
If we examine the proof of Theorem~4.10 from \cite{paperwo}, it actually proves
the following statement: Suppose that $0<\alpha<1$ and that $k\ge1$.  Further
suppose that $\sS_{1,\alpha}, \sS_{2,\alpha}, \ldots, \sS_{k,\alpha}$ are
covering sets for $B_\alpha$, $B_{2\alpha}, \ldots, B_{k\alpha}$,
respectively.  Then we can build a covering set $\sS_{k+1,\alpha}$ for
$B_{(k+1)\alpha}$ as follows:
\begin{enumerate}
\item If $k+1>2$, then for $(f,C)\in \sS_{i,\alpha}$ and $(g,D)\in
\sS_{j,\alpha}$ with $2\le i, j\le k$ and $i+j=k+2$ we include $(f\otimes g,
C+D)$ in $\sS_{k+1,\alpha}$; \\
while if $k+1=2$, then for $(f_1,C_1),(f_2,C_2),(f_3,C_3)\in \sS_{1,\alpha}$,
we include $(f_1 \otimes f_2, C_1+C_2)$ and $(f_1\otimes f_2\otimes f_3,
C_1+C_2+C_3)$ in $\sS_{2,\alpha}$.
\item For $(f,C)\in \sS_{k,\alpha}$ and any solid number $b$
with $\cpx{b}<(k+1)\alpha+3\log_3 2$, we include $(f\otimes x_1 + b, C+\cpx{b})$
in $\sS_{k+1,\alpha}$.
\item For $(f,C)\in \sS_{k,\alpha}$, any solid number $b$
with $\cpx{b}<(k+1)\alpha+3\log_3 2$, and any $v\in B_\alpha$, we include
$(v(f\otimes x_1 + b), C+\cpx{b}+\cpx{v})$ in $\sS_{k+1,\alpha}$.
\item For all $n\in T_\alpha$, we include $(n,\cpx{n})$ in $\sS_{k+1,\alpha}$.
\item For all $n\in T_\alpha$ and $v\in B_\alpha$, we include $(vn,\cpx{vn})$ in
$\sS_{k+1,\alpha}$.
\end{enumerate}

Algorithm~\ref{buildalg} is, for the most part, exactly this statement.
The only difference is the removal of the pairs $(3,3)$ and $(1,1)$ from the
possibilities of things to multiply by; this step needs additional
justification.  For $(1,1)$, this is because no number $n$ can be
most-efficiently represented as $1\cdot n$; if $(f,C)$ is a low-defect pair,
then the low-defect pair $(f,C+1)$ cannot efficiently $3$-represent anything, as
anything it $3$-represents is also $3$-represented by the pair $(f,C)$.  For
$(3,3)$, there are two possibilities.  If $3n$ is a number which is
$3$-represented by by $(3f,C+3)$, then either the representation as $3\cdot n$
is most-efficient or it is not.  If it is, then $3n$ is not a leader, and so not
in any $B_{i\alpha}$, and thus we do not need it to be $3$-represented.  If it
is not, then it is not efficiently $3$-represented by $(3f,C+3)$.  So these
particular pairs do not need to be multiplied by, and the algorithm still works.
\end{proof}

\subsection{Algorithm \ref{algcover}: Computing a covering set for $B_{r}$.}

We can now put the two of these together to form Algorithm~\ref{algcover}, for
computing a covering set for $B_r$.  If we look ahead to
Algorithm~\ref{manytruncalg}, we can turn it into a good covering.

\begin{algorithm}[ht]
\caption{Compute a covering set for $B_r$}
\label{algcover}
\begin{algorithmic}
\REQUIRE $r\in \cR, r\ge0$
\ENSURE $S$ is a covering set for $B_r$
\STATE Choose a step size $\alpha\in (0,1)\cap \cR$
\STATE Let $T_1$ be the output of Algorithm~\ref{startalg} for $\alpha$
	\COMMENT{This is a good covering of $B_\alpha$}
\FOR{$k=1$ \TO $\lceil \frac{r}{\alpha} \rceil-1$}
	\STATE Use Algorithm~\ref{buildalg} to compute a covering set $T_{k+1}$
	for $B_{(k+1)\alpha}$ from our covering sets $T_i$ for $B_{i\alpha}$
	\STATE Optional step: Do other things to $T_{k+1}$ that continue to keep
	it a covering set for $B_{(k+1)\alpha}$ while making it more practical
	to work with.  For instance, one may use Algorithm~\ref{manytruncalg} to
	turn it into a good covering of $B_{(k+1)\alpha}$, or one may remove
	elements of $T_{k+1}$ that are redundant (i.e., if one has $(f,C)$ and
	$(g,D)$ in $T_{k+1}$ such that any $n$ which is efficiently
	$3$-represented by $(f,C)$ is also efficiently represented by $(g,D)$,
	one may remove $(f,C)$)
\ENDFOR
\STATE $S \gets T_{k+1}$
\RETURN $S$
\end{algorithmic}
\end{algorithm}

\begin{proof}[Proof of correctness for Algorithm~\ref{algcover}]
Assuming the correctness of Algorithm~\ref{startalg} and
Algorithm~\ref{buildalg}, the correctness of Algorithm~\ref{algcover} follows
immediately.  Again, this is just making use of the proof of Theorem~4.10 from
\cite{paperwo}.
\end{proof}

\section{Algorithms: Computing good coverings}
\label{secalgos}

We have now completed the ``building-up'' half of the method; in this section we
will describe the ``filtering-down'' half.  The algorithms here will be based on
the proofs of the theorems in \cite{theory}, so we will once again refer the
reader to said proofs in our proofs of correctness.

\subsection{Algorithm \ref{algtrunc}: Truncating a polynomial to a given defect.}
The first step in being able to filter down is Algorithm~\ref{algtrunc}, for
truncating a given polynomial to a given defect:

\begin{algorithm}[ht]
\caption{Truncate the low-defect pair $(f,C)$ to the defect $s$}
\label{algtrunc}
\begin{algorithmic}
\REQUIRE $(f,C)$ is a low-defect pair, $s\in\cR$
\ENSURE $T$ is the truncation of $(f,C)$ to the defect $s$
\IF{$\deg f=0$}
	\IF{$\dft(f,C)<s$}
		\STATE $T\gets \{(f,C)\}$
	\ELSE
		\STATE $T\gets \emptyset$
	\ENDIF
\ELSE
	\IF{$\dft(f,C)\le s$}
		\STATE $T\gets \{(f,C)\}$
	\ELSE
		\STATE Find the smallest $K$ for which
		$\dft_{f,C}(k_1,\ldots,k_r)\ge s$, where $k_i=K+1$ if $x_i$ is
		minimal in the nesting ordering and $x_i=0$ otherwise
		\STATE $T\gets \emptyset$
		\FORALL{$x_i$ a minimal variable, $k\le K$}
			\STATE Let $g$ be $f$ with $3^k$ substituted in for
			$x_i$ and let $D=C+3k$
			\STATE Recursively apply Algorithm~\ref{algtrunc} to
			$(g,D)$ and $s$ to obtain a set $T'$
			\STATE $T\gets T\cup T'$
		\ENDFOR
	\ENDIF
\ENDIF
\RETURN $S$
\end{algorithmic}
\end{algorithm}

\begin{proof}[Proof of correctness for Algorithm~\ref{algtrunc}]
This is an algorithmic version of the method described in the proof of
Theorem~4.8 from \cite{theory}; see that for details.  (Note that $K$ is
guaranteed to exist by Proposition~\ref{minlimit}; one can find it by brute
force or slight variants.)  There is a slight difference between the two methods
in that the method described there, rather than forgetting $(f,C)$ when it
recursively applies the method to $(g,D)$ and directly generating the set $T$,
instead generates a set of values for variables that may be substituted into $f$
to yield the set $T$, only performing the substitution at the end.  This is the
same method, but without keeping track of extra information so that it can be
written in a more straightforwardly recursive manner.
\end{proof}

\subsection{Algorithm \ref{manytruncalg}: Truncating many polynomials to a given defect.}

If we can truncate one polynomial, we can truncate many of them
(Algorithm~\ref{manytruncalg}):

\begin{algorithm}[ht]
\caption{Compute a good covering of $B_r$ from a covering set for $B_r$}
\label{manytruncalg}
\begin{algorithmic}
\REQUIRE $r\in \cR$, $r\ge0$, $\sT$ a covering set for $B_r$
\ENSURE $\sS$ is a good covering of $B_r$
\STATE $\sS\gets \emptyset$
\FORALL{$(f,C)\in \sT$}
	\STATE Use Algorithm~\ref{algtrunc} to truncate $(f,C)$ to $r$; call the
		result $\sS'$
	\STATE $\sS\gets \sS\cup \sS'$
\ENDFOR
\RETURN $\sS$
\end{algorithmic}
\end{algorithm}

\begin{proof}[Proof of correctness for Algorithm~\ref{manytruncalg}]
This is an algorithmic version of the method described in the proof of
Theorem~4.9 from \cite{theory} -- that if one has a covering set for $B_r$ and
truncates each of its elements to the defect $r$, one obtains a good covering of
$B_r$.  It can also be seen as an application of the correctness
of Algorithms~\ref{buildalg} and \ref{algtrunc}.
\end{proof}

\subsection{Algorithm \ref{mainalg}: Computing a good covering of  $B_{r}$.}

We can then put this together into Algorithm~\ref{mainalg}, for computing a
good covering of $B_r$:

\begin{algorithm}[ht]
\caption{Compute a good covering of $B_r$}
\label{mainalg}
\begin{algorithmic}
\REQUIRE $r\in \cR$, $r\ge0$
\ENSURE $\sS$ is a good covering of $B_r$
\STATE Use Algorithm~\ref{algcover} to compute a covering set $\sT$ for $B_r$
\STATE Use Algorithm~\ref{manytruncalg} to compute a good covering $\sS$ for
$B_r$ from $\sT$
\RETURN $\sS$
\end{algorithmic}
\end{algorithm}

\begin{proof}[Proof of correctness for Algorithm~\ref{mainalg}]
This follows immediately from the correctness of Algorithms~\ref{algcover} and
\ref{manytruncalg}.
\end{proof}

We've now described how to compute good coverings of $B_r$.  But it still
remains to show how to use this to compute other quantities of interest, such as
$K(n)$ and $\cpx{n}_\st$.  We address this in the next section.

\section{Algorithms:  Computing stabilization length $K(n)$ and stable complexity $\cpx{n}_\st$}
\label{secmorealgos}

In order to compute $K(n)$ and $\cpx{n}_\st$, we're going to need to have to be
able to tell, algorithmically, whether, given a low-defect polynomial $f$ and a
a number $n$, there exists $k\ge 0$ such that $f$ $3$-represents $3^k n$.  If we
simply want to know whether $f$ $3$-represents $n$, this is easy; because
\[ f(3^{k_1},\ldots,3^{k_r})\ge 3^{k_1+\ldots+k_r}, \]
we have an upper bound on how large the $k_i$ can be and we can solve this with
brute force.  However, if we want to check whether it represents $3^k n$ for any
$k$, clearly this will not suffice, as there are infinitely many possibilities
for $k$.  We will need a lemma to narrow them down:

\begin{lem}
\label{maxv3}
Let $f$ be a polynomial in $r$ variables with nonnegative integer coefficients
and nonzero constant term; write
\[ f(x_1,\ldots,x_r) = \sum a_{i_1,\ldots,i_r} x_1^{i_1} \ldots x_r^{i_r} \]
with $a_{i_1,\ldots,i_r}$ positive integers and $a_{0,\ldots,0}>0$.
Let $b>1$ be a natural number and let $v_b(n)$ denote the number of times $n$
is divisible by $b$.  Then for any $k_1,\ldots,k_r\in\mathbb{Z}_{\ge0}$, we
have
\[ v_b(f(b^{k_1},\ldots,b^{k_r})) \le \sum_{a_{i_1,\ldots,i_r}>0}
	(\lfloor \log_b a_{i_1,\ldots,i_r} \rfloor+1)-1. \]
In particular, this applies when $f$ is a low-defect polynomial and $b=3$.
\end{lem}

\begin{proof}
The number $f(b^{k_1},\ldots,b^{k_r})$ is the sum of the constant term
$a_{0,\ldots,0}$ (call it simply $A_0$) and numbers of the form $A_i b^{\ell_i}$
where the $A_i$ are simply the remaining $a_{i_1,\ldots,i_r}$ enumerated in some
order (say $1\le i\le s$).  Since we can choose the order, assume that $v_b(A_1
b^{\ell_1})\le \ldots \le v_b(A_s b^{\ell_s})$.

So consider forming the number $f(b^{k_1},\ldots,b^{k_r})$ by starting with
$A_0$ and adding in the numbers $A_i b^{\ell_i}$ one at a time.  Let $S_i$
denote the sum $\sum_{j=0}^i A_j b^{\ell_j}$, so $S_0=A_0$ and
$S_s=f(b^{k_1},\ldots,b^{k_r})$.  We check that for any $i$, we have
\begin{equation}
\label{vineq}
v_b(S_i) \le \sum_{j=0}^i (\lfloor \log_b A_j \rfloor +1) - 1.
\end{equation}

Before proceeding further, we observe that if for some $i$ we have $v_b(A_{i+1}
b^{\ell_{i+1}})>v_b(S_i)$, then by assumption, for all $j>i$, $v_b(A_j
b^{\ell_j})\ge v_b(A_{i+1} b^{\ell_{i+1}})>v_b(S_i)$.  Now in general, if
$v_b(n)<v_b(m)$, then $v_b(n+m)=v_b(n)$.  So we can see by induction that for
all $j\ge i$, $v_b(S_j)=v_b(S_i)$: This is true for $j=i$, and if it is true for
$j$, then $v_b(S_j)=v_b(S_i)<v_b(A_j b^{\ell_j})$ and so
$v_b(S_{j+1})=v_b(S_i)$.

So let $h$ be the smallest $i$ such that $v_b(A_{i+1} b^{\ell_{i+1}})>v_b(S_i)$.
(If no such $i$ exists, take $h=s$.)  Then we first prove that
Equation~\eqref{vineq} holds for $i\le h$.

In the case that $i\le h$, we will in fact prove the stronger statement that
\[ \lfloor \log_b S_i \rfloor \le
	\sum_{j=0}^i (\lfloor \log_b A_j \rfloor + 1) - 1;\]
this is stronger as in general it is true that $v_b(n)\le \lfloor \log_b n
\rfloor$.  For $i=0$ this is immediate.  So suppose that this is true for $i$
and we want to check it for $i+1$, with $i+1\le h$.  Since $i+1\le h$, we have
that $v_b(A_{i+1} b^{\ell_{i+1}})\le v_b(S_i)$.  From this we can conclude the
inequality
\begin{eqnarray*}
\lfloor \log_b (A_{i+1} b^{\ell_{i+1}}) \rfloor &
= & \ell_{i+1} + \lfloor \log_b A_{i+1} \rfloor \\
\le v_b(A_{i+1} b^{\ell_{i+1}}) + \lfloor \log_b A_{i+1} \rfloor &
\le & v_b(S_i) + \lfloor \log_b A_{i+1} \rfloor.
\end{eqnarray*}

Now, we also know that
\begin{equation}
\label{twowayineq}
\lfloor \log_b S_{i+1} \rfloor \le \max\{ \lfloor \log_b S_i \rfloor,
\lfloor \log_b (A_{i+1}b^{\ell_{i+1}}) \rfloor \} + 1.
\end{equation}
And we can observe using above that
\[ \lfloor \log_b (A_{i+1} b^{\ell_{i+1}}) \rfloor + 1 \le
\lfloor \log_b S_i \rfloor + \lfloor \log_b A_{i+1} \rfloor + 1\le
\sum_{j=0}^{i+1} (\lfloor \log_b A_j \rfloor + 1) - 1.
\]
We also know that
\[ \lfloor \log_b S_i \rfloor +1 \le
\sum_{j=0}^i (\lfloor \log_b A_j\rfloor+1) \le
\sum_{j=0}^{i+1} (\lfloor \log_b A_j \rfloor +1) -1,\]
as $\lfloor \log_b A_{i+1} \rfloor + 1\ge 1$.
So we can conclude using Equation~\eqref{twowayineq} that
\[ \lfloor \log_b S_{i+1} \rfloor \le
	\sum_{j=0}^{i+1} (\lfloor \log_b A_j \rfloor + 1) - 1,\]
as desired.

Having proved Equation~\eqref{vineq} for $i\le h$, it then immediately follows
for all $i$, as by the above, for $i\ge h$,
\[ v_b(S_i)=v_b(S_h)\le
	\sum_{j=0}^h (\lfloor \log_b A_j \rfloor +1) - 1 \le
	\sum_{j=0}^s (\lfloor \log_b A_j \rfloor +1) - 1;\]
this proves the claim.
\end{proof}

\subsection{Algorithm~\ref{anykalg}: Computing whether a polynomial $3$-represents some $3^kn$.}

With this in hand, we can now write down Algorithm~\ref{anykalg} for determining
if $f$ $3$-represents any $3^k n$:

\begin{algorithm}[ht]
\caption{Determine whether $(f,C)$ $3$-represents any $3^k n$ and with what
	complexities}
\label{anykalg}
\begin{algorithmic}
\REQUIRE $(f,C)$ a low-defect pair, $n$ a natural number
\ENSURE $S$ is the set of $(k,\ell)$ such that there exist whole numbers
$(k_1,\ldots,k_r)$ with $f(3^{k_1},\ldots,3^{k_r})=3^k n$ and
$C+3(k_1+\ldots+k_r)=\ell$
\STATE $S\gets \emptyset$
\STATE Determine $v$ such that for any $k_1,\ldots,k_r$, one has
	$v_3(f(3^{k_1},\ldots,3^{k_r}))\le v$ \COMMENT{one method is given by
	Lemma~\ref{maxv3}}
\FOR{$k=0$ \TO $v-v_3(n)$}
	\FORALL{$(k_1,\ldots,k_r)$ such that $k_1+\ldots+k_r\le k+\lfloor\log_3
n\rfloor$}
		\IF{$f(3^{k_1},\ldots,3^{k_r})=3^k n$}
			\STATE $S\gets S\cup \{(k,C+3(k_1+\ldots+k_r))\}$
		\ENDIF
	\ENDFOR
\ENDFOR
\RETURN $S$
\end{algorithmic}
\end{algorithm}

\begin{proof}[Proof of correctness for Algorithm~\ref{anykalg}]
Once we have picked a $v$ (which can be found using Lemma~\ref{maxv3}), it
suffices to check if $f$ represents $3^k n$ with $k+v_3(n)\le v$.  By
Proposition~\ref{polystruct}, for any $k_1,\ldots, k_r$, we have
\[ f(3^{k_1},\ldots,3^{k_r}) \ge 3^{k_1+\ldots+k_r}, \] 
and so it suffices to check it for tuples $(k_1,\ldots,k_r)$ with
$k_1+\ldots+k_r\le \lfloor \log_3 3^k n \rfloor$.  There are only finitely many
of these and so this can be done by brute force, and this is exactly what the
algorithm does.
\end{proof}

Note that Algorithm~\ref{anykalg} is for determining specifically if there is
some $k\ge 0$ such that $f$ $3$-represents $3^k n$; it is not for $k\le 0$.  In
order to complete the algorithms that follow, we will also need to be able to
check if there is some $k\le 0$ such that $f$ $3$-represents $3^k n$.  However,
this is the same as just checking if $\hat{f}$ $3$-represents $n$, and can be
done by the same brute-force methods as were used to check if $f$ $3$-represents
$n$; no special algorithm is required here.

\subsection{Algorithm~\ref{stabalg}: Algorithm to test stability and compute stable complexity}

Now, at last, we can write down Algorithm~\ref{stabalg}, for computing $K(n)$
and $\cpx{n}_\st$.  We assume that in addition to $n$, we are given $L$, an
upper bound on $\cpx{n}$, which may be $\infty$.  Running
Algorithm~\ref{stabalg} with $L=\infty$ is always a valid choice;
alternatively, one may compute $\cpx{n}$ or an upper bound on it before applying
Algorithm~\ref{stabalg}.

\begin{algorithm}[ht]
\caption{Compute $K(n)$ and $\cpx{n}_\st$}
\label{stabalg}
\begin{algorithmic}
\REQUIRE $n$ a natural number, $L\in\N\cup\{\infty\}$, $L\ge\cpx{n}$
\ENSURE $(k,m)=(K(n),\cpx{n}_\st)$
\STATE Choose a step size $\alpha\in (0,1)\cap \cR$
\STATE Let $r$ be the smallest nonnegative integer, or $\infty$, such that
	$r\alpha>L-3\log_3 n-1$
\STATE $i \gets 1$
\STATE $U\gets \emptyset$
\WHILE{$U=\emptyset$ \AND $i\le r$}
\IF{$i=1$}
	\STATE Let $\sS_1$ be the output of Algorithm 1 for $\alpha$
	\COMMENT {This is a good covering of $B_\alpha$}
\ELSE
	\STATE Use Algorithm~\ref{buildalg} to compute a covering $\sS_i$ of
	$B_{i\alpha}$ from coverings $\sS_j$ of $B_{j\alpha}$ for $1\le j<i$
	\STATE Use Algorithm~\ref{manytruncalg} to turn $\sS_i$ into a good
		covering
\ENDIF
\STATE Optional step: Remove redundancies from $\sS_i$ as in
Algorithm~\ref{buildalg} \COMMENT{See ``optional step'' there}
\FORALL{$(f,C)\in \sS_i$}
\STATE Let $U'$ be the output of Algorithm~\ref{anykalg} on $(f,C)$ and
$n$ \COMMENT{If $r$ is finite and $i<r$ this whole loop may be skipped}
\STATE Let $s=\deg f$
\FORALL{$(k_1,\ldots,k_{s+1})$ such that $k_1+\ldots+k_{s+1}\le \lfloor\log_3
n\rfloor$}
		\IF{$\xpdd{f}(3^{k_1},\ldots,3^{k_{s+1}})=n$}
			\STATE $U'\gets U'\cup \{(k,C+3(k_1+\ldots+k_{s+1}))\}$
		\ENDIF
\ENDFOR
\STATE $U\gets U\cup U'$
\ENDFOR
\ENDWHILE
\IF{$U=\emptyset$}
\STATE $(k,m)=(0,L)$
\ELSE
\STATE Let $V$ consist of the elements $(k,\ell)$ of $U$ that minimize $\ell-3k$
\STATE Choose $(k,\ell)\in V$ that minimizes $k$
\STATE $m\gets \ell-3k$
\ENDIF
\RETURN $(k,m)$
\end{algorithmic}
\end{algorithm}

\begin{proof}[Proof of correctness for Algorithm~\ref{stabalg}]
This algorithm progressively builds up good covers $\sS_i$ of $B_{i\alpha}$ until
it finds some $i$ such that there is some $(f,C)\in \sS_i$ such that $\xpdd{f}$
$3$-represents $3^k n$ for some $k\ge 0$.  To see that this is indeed what it is
doing, observe that if \[f(3^{k_1},\ldots,3^{k_r})3^{k_{r+1}}=3^k n,\] then if
$k\ge k_{r+1}$, we may write \[f(3^{k_1},\ldots,3^{k_r})=3^{k-k_{r+1}}n\] and so
$f$ itself $3$-represents some $3^k n$, while if $k\le k_{r+1}$, we may write
\[f(3^{k_1},\ldots,3^{k_r})3^{k_{r+1}-k}=n\] and so $\xpdd{f}$ $3$-represents
$n$ itself.  And this is exactly what the inner loop does; it checks if $f$
$3$-represents any $3^k n$ using Algorithm~\ref{anykalg}, and it checks if
$\xpdd{f}$ $3$-represents $n$ using brute force.

Now, if for a given $i$ we obtain $U=\emptyset$, then that means that no $3^k n$
is $3$-represented by any $(f,C)\in \sS_i$, and so for any $k$, $\dft(3^k n)\ge
i\alpha$, that is, $\dft_\st(3^k n)\ge i\alpha$.  Conversely, if for a given $i$
we obtain $U$ nonempty, then that means that some $3^k n$ is $3$-represented by
some $(f,C)\in \sS_i$.  Since for any $(f,C)$ we have $\dft(f,C)\le i\alpha$
(and this is strict if $\deg f=0$), this means that $\dft(3^k n)<i\alpha$, and
so $\dft_\st(n)<i\alpha$.

So we see that if the algorithm exits the main loop with $U$
nonempty, it does so once has found some $i$ such that there exists $k$ with
$\dft(3^k n)< i\alpha$; equivalently, once it has found some $i$ such that
$\dft_\st(n)<i\alpha$.  Or, equivalently, once it has found some $i$ such that
$\dft(3^{K(n)} n)<i\alpha$.  Furthermore, note that $3^{K(n)} n$ must be a
leader if $K(n)>0$, as otherwise $3^{K(n)-1} n$ would also be stable.  So if
$K(n)>0$, then $3^{K(n)} n$ must be efficiently $3$-represented by some
$(f,C)\in \sS_i$.  Whereas if $K(n)=0$, then we only know that it is efficiently
$3$-represented by some $(\xpdd{f},C)$ for some $(f,C)\in \sS_i$, but we also
know $3^{K(n)}n =n$.  That is to say, the ordered pair $(K(n),\cpx{3^{K(n)} n})$
must be in the set $U$.

In this case, where $U$ is nonempty, it remains to examine the set $U$ and pick
out the correct candidate.  Each pair $(k,\ell)\in U$ consists of some $k$ and
some $\ell$ such that $\ell\ge\cpx{3^k n}$.  This implies that
\[ \dft_\st(n)\le \dft(3^k n)\le \ell-3k-3\log_3 n, \]
and so the pair $(K(n),\cpx{3^{K(n)} n})$ must be a pair $(k,\ell)$ for which
the quantity $\ell-3k-3\log_3n$, and hence the quantity $\ell-3k$, is minimized;
call this latter minimum $p$.  So
\[ \dft_\st(n) = p-3\log_3 n. \]  (Note that this means that $p=\cpx{n}_\st$.)
Then the elements of $V$ are pairs $(k,p+3k)$ with
\[\dft(3^k n)\le p-3\log_3 n,\]
but we know also that
\[\dft(3^k n)\ge \dft_\st(n) = p-3\log_3 n,\]
so we conclude that for such a pair, $\dft(3^k n)=\dft_\st(n)$.  But this means
that $3^k n$ is stable, and so $k\ge K(n)$.  But we know that $K(n)$ is among
the set of $k$ with $(k,p+3k)\in V$, and so it is their minimum.  Thus, we can
select the element $(k,\ell)\in V$ that minimizes $k$; then $k=K(n)$, and we can
take $k-3\ell$ to find $m=\cpx{n}_\st$.

This leaves the case where $U$ is empty.  In this case, we must have that for
all $1\le i\le r$, and hence in particular for $i=r$, no $(f,C)$ in $\sS_i$
$3$-represents any $n 3^k$; i.e., no $n 3^k$ lies in $B_{r\alpha}$, and hence,
by Proposition~\ref{arbr}, no $n 3^k$ lies in $A_{r\alpha}$.  That is to say,
for any $k$, $\dft(n 3^k)\ge r\alpha$, and so
\[ \cpx{n 3^k} \ge r\alpha + 3\log_3 n + 3k > L + 3k - 1.\]
Since $\cpx{n 3^k}>L+3k-1$, and $\cpx{n 3^k}\le L+3k$, we must have $\cpx{n
3^k}=L+3k$.  Since this is true for all $k\ge 0$, we can conclude that $n$ is a
stable number.  So, $n$ is stable and $\cpx{n}=L$, that is to say, $K(n)=0$ and
$\cpx{n}_\st=\cpx{n}=L$.
\end{proof}

We have now proven Theorem~\ref{frontpagethm}:

\begin{proof}[Proof of Theorem~\ref{frontpagethm}]
Algorithm~\ref{stabalg}, run with $L=\infty$, gives us a way of computing $K(n)$
and $\cpx{n}_\st$.  Then, to check if $n$ is stable, it suffices to check
whether or not $K(n)=0$.  This proves the theorem.
\end{proof}

\subsection{Algorithm~\ref{genalg}: Determining leaders and the ``drop pattern''.}

But we're not done; we can go further.  As mentioned in
Section~\ref{introdetail}, we can get more information if we go until we detect
$n$, rather than stopping as soon as we detect some $3^k n$.  We now record
Algorithm~\ref{genalg}, for not only determining $K(n)$ and $\cpx{3^{K(n)} n}$,
but for determining all $k$ such that either $k=0$ or $3^k n$ is a leader, and
the complexities $\cpx{3^k n}$.  By Proposition~\ref{arbr}, this is enough to
determine $\cpx{3^k n}$ for all $k\ge 0$.  One could also do this by using
Algorithm~\ref{stabalg} to determine $K(n)$ and then directly computing
$\cpx{3^k n}$ for all $0\le k\le K(n)$, but Algorithm~\ref{genalg} will often be
faster.

\begin{algorithm}[ht]
\caption{Compute information determining $\cpx{3^k n}$ for all $k\ge 0$}
\label{genalg}
\begin{algorithmic}
\REQUIRE $n$ a natural number, $L\in\N\cup\{\infty\}$, $L\ge\cpx{n}$
\ENSURE $V$ the set of $(k,\ell)$ where either $k=0$ or $k>0$ and $3^k n$ is a
leader, and $\ell=\cpx{3^k n}$
\STATE Choose a step size $\alpha\in (0,1)\cap \cR$
\STATE Let $r$ be the smallest nonnegative integer, or $\infty$, such that
	$r\alpha>L-3\log_3 n-1$
\STATE $i \gets 1$
\STATE $U\gets \emptyset$
\WHILE{$0\notin \pi_1(U)$, where $\pi_1$ is projection onto the first
coordinate, \AND $i\le r$}
\IF{$i=1$}
	\STATE Let $\sS_1$ be the output of Algorithm 1 for $\alpha$
	\COMMENT {This is a good covering of $B_\alpha$}
\ELSE
	\STATE Use Algorithm~\ref{buildalg} to compute a covering $\sS_i$ of
	$B_{i\alpha}$ from coverings $\sS_j$ of $B_{j\alpha}$ for $1\le j<i$
	\STATE Use Algorithm~\ref{manytruncalg} to turn $\sS_i$ into a good
	covering
\ENDIF
\STATE Optional step: Remove redundancies from $\sS_i$ as in
Algorithm~\ref{buildalg} \COMMENT{See ``optional step'' there}
\FORALL{$(f,C)\in \sS_i$}
\STATE Determine $v$ such that for any $k_1,\ldots,k_r$, one has
	$v_3(f(3^{k_1},\ldots,3^{k_r}))\le v$ \COMMENT{one method is given by
	Lemma~\ref{maxv3}} \COMMENT{If $r$ is finite and $i<r$ this whole loop
may be skipped}
\STATE Let $U'$ be the output of Algorithm~\ref{anykalg} on $(f,C)$ and $n$
\FORALL{$(k_1,\ldots,k_{r+1})$ such that $k_1+\ldots+k_{r+1}\le \lfloor\log_3
n\rfloor$}
		\IF{$\xpdd{f}(3^{k_1},\ldots,3^{k_{r+1}})=n$}
			\STATE $U'\gets U'\cup \{(k,C+3(k_1+\ldots+k_{r+1}))\}$
		\ENDIF
\ENDFOR
\STATE $U\gets U\cup U'$
\ENDFOR
\ENDWHILE
\IF{$0\notin\pi_1(U)$}
\STATE $U\gets U\cup\{(0,L)\}$
\ENDIF
\STATE Let $V=\{(k,\ell-3k): (k,\ell)\in U\}$
\STATE Let $V_m$ consist of the minimal elements of $V$ in the usual partial
order
\STATE Let $W=\{(k,p+3k): (k,p)\in V_m\}$
\RETURN $W$
\end{algorithmic}
\end{algorithm}

\begin{proof}[Proof of correctness for Algorithm~\ref{genalg}]
As in Algorithm~\ref{stabalg}, we are successively building up good coverings
$\sS_i$ of $B_{i\alpha}$, and for each one checking whether there is an
$(f,C)\in \sS_i$ and a $k\ge 0$ such that $(\xpdd{f},C)$ $3$-represents $3^k n$.
However, the exit condition on the loop is different; ignoring for a moment the
possibility of exiting due to $i>r$, the difference is that instead of stopping
once some $3^k n$ is $3$-represented, we do not stop until $n$ itself is
$3$-represented, or equivalently, $\dft(n)<i\alpha$.  We'll use $i$ here to
denote the value of $i$ when the loop exits.

We want the set $U$ to have two properties: Firstly, it should contain all the
pairs $(k,\ell)$ we want to find.  Secondly, for any $(k,\ell)\in U$, we should
have $\cpx{3^k n}\le \ell$.  For the first property, observe that if $3^k n$ is
a leader and $k>1$, then \[\dft(3^k n)\le \dft(n)-1<L-3\log_3 n-1,\] and so
$\dft(3^k n)\le r\alpha$; thus, $3^k n$ (being a leader) is efficiently
$3$-represented by some $(f,C)\in \sS_r$, and so if the loop exits due to $i>r$,
then $(k,\cpx{3^k n})\in U$.  Whereas if the loop exits due to $0\in\pi_1(U)$,
then note $\dft(3^k n)\le\dft(n)<i\alpha$, and so $3^k n$ (again being a leader)
is efficiently $3$-represented by some $(f,C)\in S_i$, and so again $(k,\cpx{3^k
n})\in U$.  This leaves the case where $k=0$.  If the loop exits due to
$0\in\pi_1(U)$, then by choice of $i$, $n$ is efficiently $3$-represented by
some $(\xpdd{f},C)$ for some $(f,C)\in \sS_i$, so $(0,\cpx{n})\in U$.  Whereas
if the loop exits due to $i>r$, then this means that $\dft(n)\ge r\alpha$, and
so
\[ \cpx{n} \ge r\alpha + 3\log_3 n > L - 1;\]
since we know $\cpx{n}\le L$, this implies $\cpx{n}=L$, and so including $(0,L)$
in $U$ means $(0,\cpx{n})\in U$.

For the second property, again, there are two ways a pair $(k,\ell)$ may end up
in $U$.  One is that some low-defect pair $(f,C)$ $3$-represents the number $3^k
n$, which, as in the proof of correctness for Algorithm~\ref{stabalg}, means
$\cpx{3^k n}\le \ell$.  The other is that $(k,\ell)=(0,L)$; but in this case,
$\cpx{n}\le L$ by assumption.

It then remains to isolate the pairs we want from the rest of $U$.  We will show
that they are in fact precisely the minimal elements of $U$ under the partial
order

\[ (k_1,\ell_1) \le (k_2,\ell_2) \iff k_1\le k_2\ \textrm{and}\ 
	\ell_1-3k_1 \le \ell_2-3k_2. \]

Say first that $(k,\ell)$ is one of the pairs we are looking for, i.e, either
$k=0$ or $3^k n$ is a leader, and $\ell=\cpx{3^k n}$.  Now suppose that that
$(k',\ell')\in U$ such that $k'\le k$ and $\ell-3k'\le \ell-3k$.  Since
$(k',\ell')\in U$, that means that $\cpx{3^{k'}n}\le \ell'$.  Since $k'\le k$,
we conclude that
\begin{equation}
\label{alg9eq}
\ell=\cpx{3^k n}\le \ell'+3(k-k')
\end{equation}
and hence that $\ell-3k \le \ell'-3k'$, so $\ell-3k=\ell'-3k'$.
Now, if $k=0$, then certainly $k\le k'$ (and so $k=k'$); otherwise, $3^k n$ is a
leader.  Suppose we had $k'<k$; then since $3^k n$ is a leader, that would mean
$\dft(3^k n) < \dft(3^{k'} n)$ and hence
\[ \cpx{3^k n} < \cpx{3^{k'}n} + 3(k-k') = \ell+3(k-k'), \]
contrary to \eqref{alg9eq}.  So we conclude $k'=k$, and so $(k,\ell)$ is indeed
minimal.

Conversely, suppose that $(k,\ell)$ is a minimal element of $U$ in this partial
order.  We must show that $\ell=\cpx{3^k n}$, and, if $k>0$, that $3^k n$ is a
leader.  Choose $k'\le k$ as large as possible with either $k'=0$ or $3^{k'}n$ a
leader, so that $\dft(3^{k'}n)=\dft(3^k n)$.  Also, let $\ell'=\cpx{3^{k'} n}$;
by above, $(k',\ell')\in U$.  Since $(k,\ell)\in U$ and $\dft(3^{k'}n)=\dft(3^k
n)$, we know that
\[ \cpx{3^{k'} n}+3(k-k')=\cpx{3^k n}\le \ell \]
and hence $\ell'-3k' \le \ell-3k$.  Since by assumption we also have $k'\le k$,
by the assumption of minimality we must have $(k',\ell')=(k,\ell)$.  But this
means exactly that either $k=0$ or $3^k n$ is a leader, and that
\[ \cpx{3^k n}=\cpx{3^{k'} n}=\ell'=\ell, \]
as needed.
\end{proof}

\subsection{Algorithm~\ref{pow2alg}: Stabilization length and stable complexity for $n=2^k$.}

Finally, before moving on to the results of applying these algorithms, we make
note of one particular specialization of Algorithm~\ref{stabalg}, namely, the
case where $n=2^k$ and $\ell=2k$.  As was noted in Section~\ref{discussion1},
this turns out to be surprisingly fast as a method of computing $\cpx{2^k}$.  We
formalize it here:

\begin{algorithm}[ht]
\caption{Given $k\ge 1$, determine $K(2^k)$ and $\cpx{2^k}_\st$}
\label{pow2alg}
\begin{algorithmic}
\REQUIRE $k\ge1$ an integer
\ENSURE $(h,p)=(K(2^k),\cpx{2^k}_\st)$
\STATE Let $(h,p)$ be the result of applying Algorithm~\ref{stabalg} with
$n=2^k$ and $L=2k$.
\RETURN $(h,p)$
\end{algorithmic}
\end{algorithm}

\begin{proof}[Proof of correctness for Algorithm~\ref{pow2alg}]
This follows from the correctness of Algorithm~\ref{stabalg} and the fact that
$\cpx{2^k}\le 2k$ for $k\ge 1$.
\end{proof}

\section{Further notes on stabilization and stable complexity}
\label{secstab}

Before we continue on to the results of applying these algorithms, let's make a
few more notes on the stabilization length $K(n)$ and the stable complexity
$\cpx{n}_\st$, now that we have demonstrated how to compute them.  We begin
with the following inequality:

\begin{prop}
\label{stabcpxineq}
For natural numbers $n_1$ and $n_2$, $\cpx{n_1 n_2}_\st \le \cpx{n_1}_\st +
\cpx{n_2}_\st$.
\end{prop}

\begin{proof}
Choose $k_1$, $k_2$, and $K$ such that $k_1+k_2=K$, both $3^{k_i} n_i$ are
stable, and $3^K n_1 n_2$ is also stable.  Then
\[ \cpx{n_1 n_2}_\st = \cpx{3^K n_1 n_2}-3K \le \cpx{3^{k_1} n_1} + \cpx{3^{k_2}
n_2} - 3(k_1+k_2) = \cpx{n_1}_\st + \cpx{n_2}_\st. \]
\end{proof}

Unfortunately, the analogous inequality for addition does not hold; for
instance,
\[ \cpx{2}_\st=2>0=\cpx{1}_\st+\cpx{1}_\st; \]
more examples can easily be found.

As was mentioned in Section~\ref{discussion2}, we can measure the instability of
the number $n$ by the quantity $\drop(n)$, defined as
\[ \drop(n)=\cpx{n}-\cpx{n}_\st=\dft(n)-\dft_\st(n). \]

We can also measure of how far from optimal a factorization is -- and, due to
Proposition~\ref{stabcpxineq}, a stabilized version:

\begin{defns}
Let $n_1,\ldots, n_k$ be positive integers, and let $N$ be their product.  We
define $\badfac(n_1,\ldots,n_k)$ to be the difference
$\cpx{n_1}+\ldots+\cpx{n_r}-\cpx{N}$.  Similarly we define
$\badfac_\st(n_1,\ldots,n_k)$ to be the difference $\cpx{n_1}_\st + \ldots +
\cpx{n_k}_\st - \cpx{N}_\st$.

If $\badfac(n_1,\ldots,n_k)=0$, we will say that the factorization
$N=n_1\cdots n_k$ is a \emph{good factorization}.  If
$\badfac_\st(n_1,\ldots,n_k)=0$, we will say that the factorization
$N=n_1\cdots n_k$ is a \emph{stably good factorization}.
\end{defns}

These definitions lead to the following easily-proved but useful equation:

\begin{prop}
\label{dropbadness}
Let $n_1,\ldots,n_k$ be natural numbers with product $N$.  Then
\[ \drop(N) + \badfac(n_1,\ldots,n_k) =
	\sum_{i=1}^k \drop(n_i) + \badfac_\st(n_1,\ldots,n_k) .\]
\end{prop}

\begin{proof}
Both sides are equal to the difference $\sum_{i=1}^k \cpx{n_i} - \cpx{N}_\st$.
\end{proof}

The usefulness of this equation comes from the fact that all the summands are
nonnegative integers.  For instance, we can obtain the following implications
from it:

\begin{cor}
Let $n_1,\ldots,n_k$ be natural numbers with product $N$; consider the
factorization $N=n_1\cdot \ldots\cdot n_k$.  Then:
\begin{enumerate}
\item If $N$ is stable and the factorization is good, then the $n_i$ are stable.
\item If the $n_i$ are stable and the factorization is stably good, then $N$ is
stable.
\item If the factorization is stably good, then $K(N)\le \sum_i K(n_i)$.
\end{enumerate}
\end{cor}

(Part (1) of this proposition also appeared as Proposition~24 in \cite{paper1}.)

\begin{proof}
For part (1), by Proposition~\ref{dropbadness}, if
$\drop(N)=\badfac(n_1,\ldots,n_k)=0$, then we must have that $\drop(n_i)=0$ for
all $i$, i.e., the $n_i$ are all stable.  For part (2), again by
Proposition~\ref{dropbadness}, if $\badfac_\st(n_1,\ldots,n_k)=0$ and
$\drop(n_i)=0$ for all $i$, then we must have $\drop(N)=0$, i.e., $N$ is stable.
Finally, for part (3) let $K_i = K(N_i)$, and let $K=K_1+\ldots+K_r$.  Then
$\prod_i (3^{K_i}n_i)=3^K n$.  Now by hypothesis,
\[\badfac_\st(3^{K_1}n_1,\ldots,3^{K_r}n_r)=\badfac_\st(n_1,\ldots,n_r)=0,\]
and furthermore each $3^{K_i}n_i$ is stable.  Hence by part (2), we must also
have that $3^K N$ is stable, that is, that $K(N)\le K=K(N_1)+\ldots+K(N_r)$.
\end{proof}

Having noted this, let us now continue on towards the results of actually
performing computations with these algorithms.

\section{Results of computation}
\label{secresults}

Armed with our suite of algorithms, we now proceed to the results of our computations.
We can use Algorithm~\ref{pow2alg} to prove Theorem~\ref{frontpage2comput}:

\begin{proof}[Proof of Theorem~\ref{frontpage2comput}]
Algorithm~\ref{pow2alg} was applied with $k=\NUM$, and it was determined that
$K(2^\NUM)=0$ and $\cpx{2^\NUM}_\st=\TWONUM$, that is to say, that $2^\NUM$ is
stable and $\cpx{2^\NUM}=\TWONUM$, that is to say, that $\cpx{2^\NUM
3^\ell}=\TWONUM+3\ell$ for all $\ell\ge 0$.  This implies that $\cpx{2^k
3^\ell}=2k+3\ell$ for all $0\le k\le \NUM$ and $\ell\ge 0$ with $k$ and $\ell$
not both zero, as if one instead had $\cpx{2^k 3^\ell}<2k+3\ell$, then writing
$2^\NUM 3^\ell = 2^{\NUM-k}(2^k 3^\ell)$, one would obtain $2^\NUM
3^\ell<\TWONUM+3\ell$.
\end{proof}

But we can do more with these algorithms than just straightforward computation
of values of complexities and stable complexities.  For instance, we can answer
the question: What is the smallest unstable defect other than $1$?

In \cite{paper1} it was determined that
\begin{thm}
\label{12d2stab}
For any $n>1$, if $\dft(n)<12\dft(2)$, then $n$ is stable.
\end{thm}

That is to say, with the exception of $1$, all defects less than $12\dft(2)$ are
stable.  This naturally leads to the question, what is the smallest unstable
defect (other than $1$)?  We might also ask, what is the smallest unstable
number (other than $1$)?  Interestingly, among unstable numbers greater than
$1$, the number $107$ turns out to be smallest both by magnitude and by defect.
However, if we measure unstable numbers (other than $1$) by their unstable
defect, the smallest will instead turn out to be $683$.  We record this in the
following theorem:

\begin{thm}
\label{smallunstab}
We have:
\begin{enumerate}
\item The number $107$ is the smallest unstable number other than $1$.
\item Other than $1$, the number $107$ is the unstable number with the smallest
defect, and $\dft(107)=3.2398\ldots$ is the smallest unstable defect other than
$1$.
\item Among nonzero values of $\dft_\st(n)$ for unstable $n$, $\dft_\st(683)$,
or $\dft(2049)=2.17798\ldots$, is the smallest.
\end{enumerate}
\end{thm}

\begin{proof}
For part (1), it suffices to use Algorithm~\ref{stabalg} to check the stability
of all numbers from $2$ to $106$.

For parts (2) and (3), in order to find unstable numbers of small defect, we
will search for leaders of small defect which are divisible by $3$.  (Since if
$n$ is unstable, then $3^{K(n)} n$ is a leader divisible by $3$, and
$\dft(3^{K(n)} n)<\dft(n)$).  We use Algorithm~\ref{mainalg} to compute a good
covering $S$ of $B_{21\dft(2)}$.  Doing a careful examination of the low-defect
polynomials that appear, we can determine all the multiples of $3$ that each one
can $3$-represent; we omit this computation, but its results are that the
following multiples of $3$ can be $3$-represented:
$3$, $6$, $9$, $12$, $15$, $18$, $21$, $24$, $27$, $30$, $33$, $36$, $39$, $42$,
$45$, $48$, $54$, $57$, $60$, $63, 66$, $72$, $75$, $78$, $81$, $84$, $90$,
$96$, $111$, $114$, $120$, $126$, $129$, $132$, $144$, $162$, $165, 168$, $171$,
$180$, $192$, $225$, $228$, $231$, $240$, $252$, $258$, $264$, $288$, $321$,
$324$, $330$, $336, 360$, $384$, $480$, $513$, $516$, $528$, $576$, $768$,
$1026$, $1032$, $1056$, $1152$, $1536$, $2049$, $2052$, $2064$, $2112$, $2304$,
$3072$, and, for $k\ge 0$, numbers of the forms $12\cdot 3^k+3$,
$6\cdot 3^k+3$, $9\cdot3^k+3$, $12\cdot3^k+6$, and $18\cdot3^k+6$.

For the individual leaders, we can easily check by computation that the only
ones which are leaders are $3$, $321$, and $2049$.  This leaves the infinite
families.  For these, observe that if we divide them by $3$, we get,
respectively, $4\cdot 3^k+1$, $2\cdot 3^k+1$, $3\cdot3^k+1$, $2(2\cdot3^k+1)$,
and $2(3\cdot3^k+1)$, and it is easy to check that any number of any of those
forms has defect less than $12\dft(2)$ and hence is stable by
Theorem~\ref{12d2stab}; thus, multiplying them by $3$ cannot yield a leader.

So we conclude that the only leaders $m$ with $\dft(m)<21\dft(2)$ are $3$,
$321$, and $2049$.  Therefore, the only unstable numbers $n$ with
$\dft_\st(n)<21\dft(2)$ are $1$, $107$, and $683$.  Note also that by the above
computation, no power of $3$ times any of $3$, $321$, or $2049$ is a leader (as
it would have to have smaller defect and would thus appear in the list), and
thus the numbers $3$, $321$, and $2049$ are not just leaders but in fact stable
leaders.  So to prove part (3), it suffices to note that, since $\dft_\st(3)=0$,
among $\dft_\st(107)$ (i.e.~$\dft(321)$) and $\dft_\st(683)$
(i.e.~$\dft(2049)$), the latter is smaller.

This leaves part (2).  Observe that $\dft(107)=\dft(321)+1$.  And if $n$ is
unstable, then $\dft_\st(n)\le \dft(n)-1$.  So if $n>1$ is unstable and
$\dft(n)<\dft(107)$, then $\dft_\st(n)<\dft(321)$, which by the above forces
$n=683$.  But in fact, although $\dft(2049)<\dft(107)$, we nonetheless have
$\dft(683)>\dft(107)$ (because while $\dft(107)=\dft(321)+1$,
$\dft(683)=\dft(2049)+2$).  Thus $\dft(107)$ is the smallest unstable defect
other than $1$, i.e., $107$ is (other than $1$) the smallest unstable number by
defect.
\end{proof}

These computational results  provide a good demonstration of the power of the
methods here.

\subsection*{Acknowledgements}
The author is grateful to J.~Arias de Reyna for helpful discussion. He thanks
his advisor J.~C.~Lagarias for help with editing and further discussion.
Work of the author was supported by NSF grants DMS-0943832 and DMS-1101373.

\appendix

\section{Implementation notes}
\label{impnotes}

In this appendix we make some notes about the attached implementation of the
above algorithms and on other ways they could be implemented.

We have actually not implemented Algorithm~\ref{stabalg} and
Algorithm~\ref{genalg} in full generality, where $L$ may be arbitrary; we have
only implemented the case where $L=\infty$, the case where $L=\cpx{n}$ (computed
beforehand), and the case of Algorithm~\ref{pow2alg}.

As was mentioned in Section~\ref{discussion1}, the step size in the attached
implementation has been fixed at $\alpha=\dft(2)$, with the sets $B_\alpha$ and
$T_\alpha$ precomputed.  Other integral multiples of $\dft(2)$ were tried, up to
$9\dft(2)$ (since $10\dft(2)>1$ and thus is not a valid step size), but these
all seemed to be slower, contrary to the author's expectation.

Another variation with a similar flavor is that one could write a version of
these algorithms with nonstrict inequalities, computing numbers $n$ with
$\dft(n)\le r$ for a given $r$, rather than $\dft(n)<r$; see Appendix~A of
\cite{theory}.  We may define:
\begin{defn}
For a real number $r\ge 0$, the set $\overline{A}_r$ is the set $\{ n\in\N:
\dft(n)\le r\}$.  The set $\overline{B}_r$ is the set of all elements of
$\overline{A}_r$ which are leaders.
\end{defn}
\begin{defn}
A finite set $\sS$ of low-defect pairs will be called a \emph{covering set} for
$\overline{B}_r$ if, for every $n\in \overline{B}_r$, there is some low-defect
pair in $\sS$ that efficiently $3$-represents it.  We will say $\sS$ is a
\emph{good covering} of $\overline{B}_r$ if, in addition, every $(f,C)\in\sS$
satisfies $\dft(f,C)\le r$.
\end{defn}
Then, as per Appendix~A of \cite{theory}, good coverings of $\overline{B}_r$
exist, and only slight variations on the algorithms above are needed to compute
them.  However, this was not tried in this implementation.

It is also worth noting that the check for whether a given polynomial $f$
$3$-represents a given number $n$ can also be sped up.  If $f$ is a low-defect
polynomial with leading coefficent $a$, maximum coefficient $A$, and $N$ terms,
then
\[ a3^{k_1+\ldots+k_r} \le f(3^{k_1},\ldots,3^{k_r}) \le NA3^{k_1+\ldots+k_r},\]
so we only need to search $(k_1,\ldots,k_r)$ with
\[ \lceil \log_3 \frac{n}{NA} \rceil \le k_1+\ldots+k_r \le
	\lfloor \log_3 \frac{n}{a} \rfloor, \]
a stricter condition than was described in the algorithms above.  This
improvement is, in fact, used in the attached implementation.  It is also
possible that there is a better way than brute force.

As was mentioned in Section~\ref{secmorealgos}, when running
Algorithm~\ref{stabalg} or Algorithm~\ref{genalg} with $L$ finite, one can omit
the $3$-representation check at intermediate steps.  We have only implemented
this variant for Algorithm~\ref{pow2alg}.
 
It was mentioned in Section~\ref{bdsec} that considering ``low-defect expression
pairs'' $(E,C)$ or ``low-defect tree pairs'' $(T,C)$ (where $E$ is a low-defect
expression, $T$ is a low-defect tree, and $C\ge\cpx{E}$ or $C\ge\cpx{T}$, as
appropriate) may be useful.  In fact, the attached implementation works with a
tree representation essentially the same as low-defect trees and low-defect tree
pairs.  Among other things, this makes it easy to find the minimal variables to
be substituted into.  If one were actually representing low-defect polynomials
as polynomials, this would take some work.  There is a slight difference in
that, rather than simply storing a base complexity $C\ge \cpx{T}$, it stores for
each vertex or edge -- say with label $\cpx{n}$ -- a number $k$ such that $k\ge
\cpx{n}$, unless we are talking about a non-leaf vertex and $n=1$, in which case
$k=0$.  We can then determine a $C$ by adding up the values of $k$
That is to say, the complexity, rather than being
attributed to the whole tree, is distributed among the parts of the tree
responsible for it; this makes it easier to check for and remove redundant
low-defect pairs.

It was also mentioned in Section~\ref{bdsec} that one could use a representation
similar to low-defect expressions, but with all the integer constants replaced
with $+,\cdot,1$-expressions for same.  E.g., instead of $2(2x+1)$, one might
have $(1+1)((1+1)x+1)$.  We have not implemented this, but doing this woud have
one concrete benefit: It would allow the algorithms above to not only determine
the complexity of a given number $n$, but also to give a shortest
representation.  (And analogously with stable complexity.)  The current
implementation cannot consistently do this in a useful manner.  For instance,
suppose that we ran Algorithm~\ref{pow2alg} and found some $k$ with
$\cpx{2^k}=2k-1$.  We might then look at the actual low-defect pair $(f,C)$ that
$3$-represented it, to learn what this representation with only $2k-1$ ones is.
But it might turn out, on inspection, that $f$ was simply the constant $2^k$;
this would not be very enlightening.  Using $+,\cdot,1$-expressions would remedy
this, as would having low-defect pairs keep track of their ``history'' somehow.

It's also possible to write numerical versions of Proposition~\ref{minlimit},
that say exactly how far out one has to go in order to get within a specified
$\varepsilon$ of the limit $\dft(f,C)$; one could use this in
Algorithm~\ref{algtrunc} instead of simply searching larger and larger $K$ until
one works.  This was tried but found to be slower.

Finally, it is worth expanding here on the remark in Section~\ref{discussion2}
that it is possible to write Algorithm~\ref{stabalg} and
Algorithm~\ref{genalg} without using truncation.  Surprisingly little
modification is required; the only extra step needed is that, in order to check
if $n$ (or any $3^k n$) has defect less than $i\alpha$, instead of just checking
if a low-defect pair $(f,C)$ (or its augmented version) $3$-represents $n$ (or
any $3^k n$), if one finds that indeed $n=f(3^{k_1},\ldots,3^{k_r})$ (or the
appropriate equivalent), one must additionally check whether
$\dft_{f,C}(k_1,\ldots,k_r)<i\alpha$, since this is no longer guaranteed in
advance.  We will not state a proof of correctness here; it is similar to the
proofs above.  Such no-truncation versions of the algorithms were tried, but
found to be too slow to be practical, because of the time needed to check
whether the resulting polynomials $3$-represented a given number.  Another
possibility, in the case where one is using a cutoff, is to truncate only at the
final step, and not at the intermediate steps; this has not been tried.  If this
is used, it should probably be combined with not checking whether $n$ (or any
$3^k n$) is $3$-represented until the final step, for the reason just stated.

\end{document}